\documentclass{article}

\usepackage{setspace} 
\usepackage[backend=biber, style=alphabetic, sorting=nyt, maxalphanames=99, minalphanames=99, maxbibnames=99]{biblatex}

\usepackage{amsmath}
\usepackage{amsthm}
\usepackage{amssymb}
\usepackage{enumitem}
\usepackage{tikz-cd}
\usepackage{mathtools}
\usepackage{xcolor}
\usepackage{import}
\usepackage{graphicx}
\usepackage{mathrsfs}
\usepackage[a4paper, portrait]{geometry}

\usepackage[colorlinks=true, linkcolor=blue]{hyperref}

\usepackage{cleveref}

\newtheorem{theorem}{Theorem}[section]
\newtheorem{lemma}[theorem]{Lemma}
\newtheorem{corollary}[theorem]{Corollary}
\newtheorem{proposition}[theorem]{Proposition}

\theoremstyle{definition}
\newtheorem{definition}[theorem]{Definition}

\newtheorem{remark}[theorem]{Remark}
\newtheorem{example}[theorem]{Example}

\crefname{lemma}{lemma}{lemmas}
\Crefname{lemma}{Lemma}{Lemmas}
\crefname{corollary}{corollary}{corollaries}
\Crefname{corollary}{Corollary}{Corollaries}
\crefname{proposition}{proposition}{propositions}
\Crefname{proposition}{Proposition}{Propositions}
\crefname{definition}{definition}{definitions}
\Crefname{definition}{Definition}{Definitions}
\crefname{notation}{notation}{notations}
\Crefname{notation}{Notation}{Notations}
\crefname{remark}{remark}{remarks}
\Crefname{remark}{Remark}{Remarks}
\crefname{example}{example}{examples}
\Crefname{example}{Example}{Examples}
\crefname{question}{question}{questions}
\Crefname{question}{Question}{Questions}

\newcommand{\cat}[1]{\mathcal{#1}}
\newcommand{\hocat}[1]{\cat{K}^{#1}}
\newcommand{\Hom}[1]{\mathrm{Hom}_{#1}}
\newcommand{\Ext}[1]{\mathbb{E}_{#1}}
\newcommand{\Extg}[1]{\mathrm{Ext}_{#1}}
\newcommand{\inflation}{\rightarrowtail}
\newcommand{\deflation}{\twoheadrightarrow}

\newcommand{\proj}{\mathrm{proj}}
\newcommand{\inj}{\mathrm{inj}}
\newcommand{\add}{\mathrm{add}}
\newcommand{\mo}{\mathrm{mod}\ }

\newcommand{\id}[1]{\mathrm{id}(#1)}
\newcommand{\cone}[1]{\mathrm{cone}(#1)}
\newcommand{\idd}[1]{\mathrm{idd}(#1)}

\newcommand{\pd}[1]{\mathrm{pd}(#1)}

\newcommand{\xinflation}[1]{\stackrel{#1}{\rightarrowtail}}
\newcommand{\xdeflation}[1]{\stackrel{#1}\twoheadrightarrow}
\newcommand{\mat}[1]{\left[\begin{smallmatrix}#1\end{smallmatrix}\right]}

\title{Extriangulated ideal quotients and \(d\)-Auslander categories}
\author{Lior Silberberg}
\date{}

\begin{document}

	\maketitle
	
	\begin{abstract}
		Building on recent studies of 0-Auslander categories, we establish a connection between \(d\)-Auslander extriangulated categories and categories of \((d+2)\)-term complexes up to homotopy. 
		We give a precise homological condition under which an algebraic extriangulated category admits an extriangulated ideal quotient equivalent to \(\mathcal{K}^{[-d-1,0]}(\mathcal{A})\). 
		We then demonstrate that \(d\)-cluster-tilting subcategories in triangulated categories serve as a key source of \(d\)-Auslander extriangulated categories. 
		Using these structural results, we answer a question posed by Iyama \cite[Appendix A]{Yang}  by proving that \(\mathcal{K}^{[-d-1,0]}(\mathcal{N})\) admits a triangulated structure when \(\mathcal{N}\) is a weakly idempotent complete algebraic \((d+4)\)-angulated category.
	\end{abstract}
	
	\section{Introduction}
	
	Extriangulated categories were introduced by Nakaoka and Palu in \cite{NP} as a simultaneous generalization of exact and of triangulated categories. In the recent paper \cite{FGPPP}, the authors introduce the notion of extriangulated ideal quotient, which, roughly speaking, is an additive quotient which does not alter the extriangulated structure. For any extriangulated category \(\cat{C}\), they consider the ideal \(J\) of morphisms which factor through some morphism \(f : I \rightarrow P\) with an injective domain and a projective codomain. They then show that the ideal quotient \(\cat{C}/J\) has a natural extriangulated structure, and that the quotient functor \(\cat{C} \rightarrow \cat{C}/{J}\) is an extriangulated ideal quotient. \par 
	Inspired by the results of \cite{BrustleYang}, they use this construction to study algebraic \(0\)-Auslander extriangulated categories, which are closely related to cluster categories. Their main result, which was also obtained by \cite{Chen} and \cite{Yang}, is the following.
	\begin{theorem}\label{theorem-FGPPP}
		Let \(\cat{C}\) be an algebraic \(0\)-Auslander extriangulated category. Then there exists an additive category \(\cat{A}\) and an equivalence of extriangulated categories
		\begin{align*}
			\cat{C}/{J} \simeq \hocat{[-1,0]}(\cat{A}),
		\end{align*}
		where \(\hocat{[-1,0]}(\cat{A})\) is the category of two term complexes in \(\cat{A}\) up to homotopy.
	\end{theorem}
	It is thus natural to study the connection between categories of \((d+2)\)-term complexes up to homotopy and (\(d\)-Auslander) extriangulated categories. Using an explicit embedding of algebraic extriangulated categories inside a triangulated category, we can give a precise relation between the aforementioned notions.
	\begin{theorem}[\Cref{theorem-ideal-quotients}]
		\label{theorem-introduction}
		Let \(\cat{C}\) be an algebraic extriangulated category. Denote by \(\cat{P}\), \(\cat{I}\) and \(\cat{Q}\) its subcategories of projective, injective, and projective-injective objects, respectively. The following statements are equivalent:
		\begin{enumerate}
			\item \(\cat{C}\) admits an extriangulated ideal quotient which is equivalent as an extriangulated category to \(\hocat{[-d-1,0]}(\cat{A})\) for some additive category \(\cat{A}\).
			\item \(\cat{C}\) is \(d\)-Auslander and satisfies \(\Ext{\cat{C}}^k(I,P)=0\) for any \(1\leq k \leq d\), \(I \in \cat{I}\) and \(P \in \cat{P}\).
		\end{enumerate}
		When these equivalent conditions hold, we have \(\cat{A} \simeq \cat{P}/[\cat{Q}]\) and \(\cat{C}/[\cat{I}\rightarrow \cat{P}] \simeq \hocat{[-d-1,0]}(\cat{A})\). 
	\end{theorem}
	Clearly, this theorem specializes to \Cref{theorem-FGPPP} in the case \(d=0\). From another point of view, \Cref{theorem-introduction} can also be seen as a complete homological characterization of truncated homotopy categories among the class of algebraic extriangulated categories. \par
	One source of \(d\)-Auslander extriangulated categories is cluster-tilting subcategories in triangulated categories. Namely, If \(\cat{M}\) is a \((d+2)\)-cluster-tilting subcategory of a triangulated category \(\cat{T}\), we can consider the relative extriangulated structure on \(\cat{T}\) which makes \(\cat{M}\) the subcategory of projective objects. Denote by \(\cat{T}_\cat{M}\) the category \(\cat{T}\) with this extriangulated structure. We show that \(\cat{T}_\cat{M}\) is \(d\)-Auslander. Applying \Cref{theorem-introduction} in this setting, we show:
	\begin{theorem}[\Cref{corollary-CT-vosnex,corollary-dZCT}]
		\label{theorem-intro-triangulated}
		Let \(\cat{T}\) be an algebraic triangulated category, and let \(\cat{M} \subseteq \cat{T}\) be a \((d+2)\)-cluster-tilting subcategory. 
		\begin{enumerate}
			\item If \(\cat{M}\) satisfies the vosnex property, i.e. \(\Hom{\cat{T}}(\cat{M},\Sigma^{-i}\cat{M}) = 0\) for all \(1 \leq i \leq d\), then there is an equivalence of extriangulated categories
			\begin{align*}
				\cat{T}_\cat{M}/[\Sigma^{d+1}\cat{M} \rightarrow \cat{M}] \simeq \hocat{[-d-1,0]}(\cat{M}).
			\end{align*}
			\item If \(\cat{M}\) is \((d+2)\mathbb{Z}\)-cluster-tilting, i.e. \(\Sigma^{d+2}\cat{M} = \cat{M}\), then there is an equivalence of extriangulated categories
			\begin{align*}
				\cat{T}_\cat{M} \simeq \hocat{[-d-1,0]}(\cat{M}).
			\end{align*}
		\end{enumerate}
	\end{theorem}
	Combining \Cref{theorem-intro-triangulated} with \cite[Theorem 7.5]{Kvamme}, we obtain the following corollary.
	\begin{corollary}[\Cref{cor-angulated}]
		\label{cor-intro}
		Let \(\cat{N}\) be a weakly idempotent complete algebraic \((d+4)\)-angulated category. Then \(\hocat{[-d-1,0]}(\cat{N})\) has a structure of a triangulated category. 
	\end{corollary}
	In the case \(d=0\), this gives an answer to a question of Iyama, which appears in \cite[Appendix A]{Yang} . We also give a generalization of \cite[Theorem A.1]{Yang}. \par
	Furthermore, we apply \Cref{theorem-introduction} in the abelian setting. In particular, we recover \cite[Theorem 3.1]{GW}. \par
	During the finalization of this paper, we learned that Zhaotai Zhang, Yu Zhou, and Bin Zhu independently proved similar results in \cite{ZZZ}, which appeared on arXiv on the same day as this preprint. Concretely, our \Cref{theorem-introduction} is essentially equivalent to the combination of \cite[Theorems 1.1 and 1.2]{ZZZ}, though approached from different perspective. Furthermore, \Cref{theorem-intro-triangulated} and \cite[Corollary 1.3]{ZZZ} yield the same result. It is worth noting that despite these parallel foundational results, the general structure and point of view of the two papers differ, as do their main applications.
	
	\subsection*{Structure of the paper}
	
	In \Cref{preliminaries}, we give necessary background, as well as a direct and explicit construction of a triangulated hull of an extriangulated category with enough projectives. \Cref{d-Auslander} is devoted to studying the homological properties of algebraic \(d\)-Auslander extriangulated categories. In \Cref{Main}, we give the proof of \Cref{theorem-introduction}. In \Cref{triangulated}, we apply the results of \Cref{Main} to study \(d\)-Auslander extriangulated categories which come from triangulated categories with cluster-tilting subcategories. In particular, we prove \Cref{theorem-intro-triangulated} and \Cref{cor-intro}. Finally, \Cref{higher} applies the results of \Cref{Main} in the abelian setting. 
	
	\subsection*{Acknowledgments}

	I would like to thank my advisors Yann Palu and David Hernandez for their guidance and support. This project has received funding from the European Union’s Horizon Europe research and innovation programme under the Marie Skłodowska-Curie grant agreement n° 101126554. \par 
	Co-Funded by the European Union. Views and opinions expressed are however those of the author only and do not necessarily reflect those of the European Union. Neither the European Union nor the granting authority can be held responsible for them. \def\svgwidth{3.6em}
\begingroup%
  \makeatletter%
  \providecommand\color[2][]{%
    \errmessage{(Inkscape) Color is used for the text in Inkscape, but the package 'color.sty' is not loaded}%
    \renewcommand\color[2][]{}%
  }%
  \providecommand\transparent[1]{%
    \errmessage{(Inkscape) Transparency is used (non-zero) for the text in Inkscape, but the package 'transparent.sty' is not loaded}%
    \renewcommand\transparent[1]{}%
  }%
  \providecommand\rotatebox[2]{#2}%
  \newcommand*\fsize{\dimexpr\f@size pt\relax}%
  \newcommand*\lineheight[1]{\fontsize{\fsize}{#1\fsize}\selectfont}%
  \ifx\svgwidth\undefined%
    \setlength{\unitlength}{2167.83728027bp}%
    \ifx\svgscale\undefined%
      \relax%
    \else%
      \setlength{\unitlength}{\unitlength * \real{\svgscale}}%
    \fi%
  \else%
    \setlength{\unitlength}{\svgwidth}%
  \fi%
  \global\let\svgwidth\undefined%
  \global\let\svgscale\undefined%
  \makeatother%
  \begin{picture}(1,0.18682214)%
    \lineheight{1}%
    \setlength\tabcolsep{0pt}%
    \put(0,0){\includegraphics[width=\unitlength,page=1]{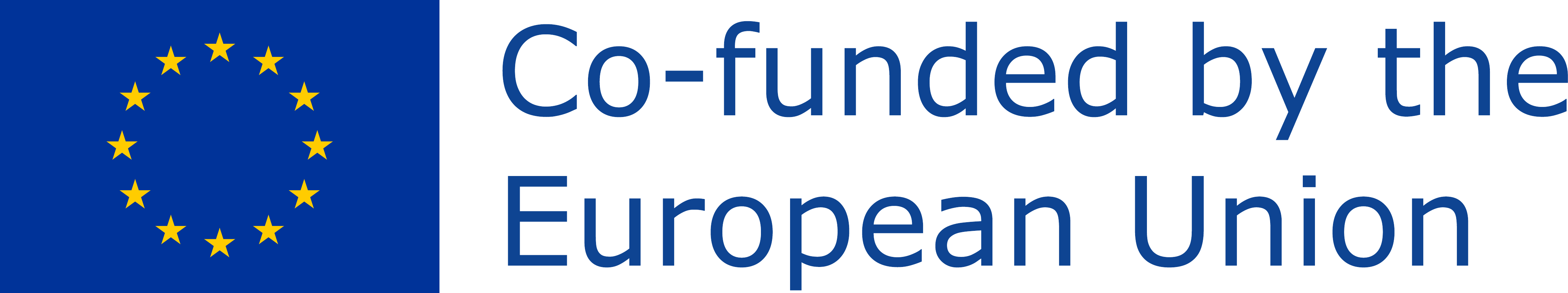}}%
  \end{picture}%
\endgroup%

	\section{Preliminaries}\label{preliminaries}
	
	\subsection{Extriangulated notions}
	
	We assume the reader is familiar with extriangulated categories. More details can be found in \cite{NP}.
	\begin{definition}
		An extriangulated functor between two extriangulated categories \((\cat{C}_1,\mathbb{E}_1,\mathfrak{s}_1)\) and \((\cat{C}_2,\mathbb{E}_2,\mathfrak{s}_2)\) is an additive functor \(F : \cat{C}_1 \rightarrow \cat{C}_2\) together with a natural transformation \(\alpha : \mathbb{E}_1 \rightarrow \mathbb{E}_2 \circ (F^{\mathrm{op}}\times F)\), satisfying \(\mathfrak{s}_2(\alpha(\delta)) = F(\mathfrak{s}_1(\delta))\).
	\end{definition}
	\begin{definition}
		Let \((\cat{C}_1,\mathbb{E}_1,\mathfrak{s}_1)\) and \((\cat{C}_2,\mathbb{E}_2,\mathfrak{s}_2)\) be extriangulated categories. The extriangulated functor \((F,\alpha)\) is an extriangulated ideal quotient if \(F: \cat{C}_1 \rightarrow \cat{C}_2\) is an ideal quotient and \(\alpha\) is the identity. 
	\end{definition}
	\begin{example}
		Let \(\cat{C}\) be an extriangulated category, and let \(\cat{Q}\) be a subcategory of projective-injective objects. Then \(\cat{C}/[\cat{Q}]\) is extriangulated, and the quotient functor \(\cat{C} \rightarrow \cat{C}/[\cat{Q}]\) is an extriangulated ideal quotient. 
	\end{example}
	We will often only write \(\cat{C}\) for an extriangulated category, and \(F:\cat{C}_1 \rightarrow \cat{C}_2\) for an extriangulated functor. 
	\begin{theorem}[Theorem 2.8 in \cite{FGPPP}, Proposition 3.16 in \cite{ChenDG}]
		Let \(\cat{C}\) be an extriangulated category. Let \(J_0\) be a class of morphisms with an injective domain and a projective codomain, and let \(J\) be the ideal generated by \(J_0\). The ideal quotient \(\cat{C}/J\) has a natural extriangulated structure. With this extriangulated structure, the quotient functor \(\cat{C} \rightarrow \cat{C}/J\) is an extriangulated ideal quotient. 
	\end{theorem}
	
	An important idea in the theory of extriangulated categories is the possibility to replace an extriangulated structure by a substructure. Namely
	\begin{definition}
		A relative structure on an extriangulated category \((\cat{C},\mathbb{E},\mathfrak{s})\) is an extriangulated category \((\cat{C},\mathbb{F},\mathfrak{t})\) such that \(\mathbb{F} \subseteq \mathbb{E}\), and \(\mathfrak{t} = \mathfrak{s}|_{\mathbb{F}}\).
	\end{definition}
	
	One way to get relative structure on a given extriangulated category which will be essential in \Cref{triangulated} is the following.
	
	\begin{proposition}[Proposition 3.19 in \cite{HLN}]
		Let \((\cat{C},\mathbb{E},\mathfrak{s})\) be an extriangulated category, and let \(\cat{M} \subseteq \cat{C}\) be a full additive subcategory. For \(A,C \in \cat{C}\), let \(\mathbb{E}_\cat{M}\) be the subfunctor of \(\mathbb{E}\) given by
		\begin{align*}
			\mathbb{E}_\cat{M}(C,A) \coloneq \{ \delta \in \mathbb{E}(C,A) \mid f^*\delta = 0 \text{ for all } f \in \Hom{\cat{C}}(S,C) \text{ and any } S \in \cat{M}\}.
		\end{align*}
		Then \((\cat{C},\mathbb{E}_{\cat{M}}, \mathfrak{s}|_{\mathbb{E}_\cat{M}})\) is an extriangulated category. 
	\end{proposition}
	
	Finally, we give the definition \(d\)-Auslander extriangulated categories, which is a key notion throughout the paper.
	
	\begin{definition}
		Let \(d \in \mathbb{Z}_{\geq 0}\) be a non-negative integer. An extriangulated category \(\cat{C}\) is said to be \(d\)-Auslander if the following properties hold.
		\begin{enumerate}
			\item The category \(\cat{C}\) has enough projectives. 
			\item The global dimension of \(\cat{C}\) is at most \(d+1\). That is, for every \(X\in \cat{C}\) and \(1 \leq i \leq d+1\), there are conflations
			\begin{align*}
				X_i \inflation P_i \deflation X_{i-1}
			\end{align*}
			such that \(X_0 = X\), the middle term \(P_i\) is projective for all \(1 \leq i \leq d+1\), and \(X_{d+1}\) is projective. 
			\item The dominant dimension of \(\cat{C}\) is at least \(d+1\). That is, for any projective \(P \in \cat{C}\) and \(1 \leq i \leq d+1\), there are conflations
			\begin{align*}
				X_i \inflation Q_i \deflation X_{i+1}
			\end{align*}
			such that \(X_1 = P\), the middle term \(Q_i\) is projective-injective for all \(1 \leq i \leq d+1\), and \(X_{d+2}\) is injective. 
		\end{enumerate} 
	\end{definition}
	
	\subsection{Cluster-tilting subcategories}
	We now would like to recall some of the definitions and basic properties of \(d\)-cluster-tilting subcategories, which will be necessary in \Cref{triangulated}.
	\begin{definition}
		Let \(\cat{T}\) be a triangulated category. A subcategory \(\cat{M}\) is called \(d\)-cluster-tilting if it is functorially finite, and satisfies
		\begin{enumerate}
			\item \(\cat{M} = \{X \in \cat{T} \mid \Hom{\cat{T}}(X,\Sigma^i M)=0 \text{ for all } M \in \cat{M} \text{ and } 1 \leq i < d\}\),
			\item \(\cat{M} = \{X \in \cat{T} \mid \Hom{\cat{T}}(M,\Sigma^i X)=0 \text{ for all } M \in \cat{M} \text{ and } 1 \leq i < d\}\).
		\end{enumerate}
	\end{definition}
	For two subcategories \(\cat{X},\cat{Y} \subseteq \cat{T}\), we denote by \(\cat{X}\ast \cat{Y}\) the full subcategory of \(\cat{T}\) with the objects \(T \in \cat{T}\) that fit in a triangle
	\begin{align*}
		X \rightarrow T \rightarrow Y \rightarrow \Sigma X,
	\end{align*}
	with \(X \in \cat{X}\) and \(Y \in \cat{Y}\). The octahedral axiom implies that \((\cat{X}\ast \cat{Y}) \ast \cat{Z} = \cat{X} \ast (\cat{Y}\ast \cat{Z})\) for any \(\cat{X},\cat{Y},\cat{Z} \subseteq \cat{T}\).
	
	\begin{theorem}[Theorem 3.1 in \cite{IyamaYoshino}]
		Let \(\cat{M}\) be a \(d\)-cluster-tilting subcategory in a triangulated category \(\cat{T}\). Then \(\cat{T} = \cat{M}\ast \Sigma\cat{M} \ast \cdots \ast \Sigma^{d-1}\cat{M}\).
	\end{theorem}
	
	\begin{corollary}
		Let \(\cat{M}\) be a \(d\)-cluster-tilting subcategory in a triangulated category \(\cat{T}\). Let \(f : X \rightarrow Y\) be a morphism in \(\cat{T}\) such that the map
		\begin{align*}
			\Hom{\cat{T}}(M,X) \xrightarrow{f\circ} \Hom{\cat{T}}(M,Y)
		\end{align*}
		vanishes for all \(M \in \cat{M}\). Then \(f \in [\Sigma\cat{M}\ast\cdots\ast\Sigma^{d-1}\cat{M}]\).
	\end{corollary}
	\begin{proof}
		According to the above theorem, \(X\) fits in a triangle of the form
		\begin{align*}
			M \xrightarrow{g} X \xrightarrow{h} N \rightarrow \Sigma M
		\end{align*}
		with \(M \in \cat{M}\) and \(N \in \Sigma\cat{M}\ast\cdots\ast\Sigma^{d-1}\cat{M}\). Since the composition \(f \circ g\) vanishes, the morphism \(f\) factors through the morphism \(h\), so \(f \in [\Sigma\cat{M}\ast\cdots\ast\Sigma^{d-1}\cat{M}]\).
	\end{proof}
	One last definition that will be useful for us, and that appears in the works of \cite{KellerReiten, IO13} is the following.
	\begin{definition}
		A \(d\)-cluster-tilting subcategory \(\cat{M}\) of a triangulated category \(\cat{T}\) is said to satisfy the vosnex property (vanishing of small negative extensions) if
		\begin{align*}
			\Hom{\cat{T}}(M,\Sigma^{-i}M') = 0
		\end{align*} 
		for all \(M, M' \in \cat{M}\) and \(1\leq i \leq d-2\). 
	\end{definition}
	
	\subsection{A triangulated hull}
	
	In this subsection we will give an explicit description of a triangulated hull for algebraic extriangulated categories with enough projective objects. This is a particular case of \cite[Proposition-Definition 3.14]{ChenDG}. \par
	Let \(\cat{C}\) be an algebraic extriangulated category, and let \(\cat{E}\) be an exact category with a subcategory \(\cat{Q}_0\) of projective-injective objects satisfying \(\cat{C} \simeq \cat{E}/[\cat{Q}_0]\). We also set \(\cat{P} = \proj(\cat{E})\).
	We will consider the category \(\hocat{-}(\cat{P})\) and its triangulated subcategory \(\hocat{b}(\cat{Q}_0)\). We let \(\iota : \cat{E} \rightarrow \hocat{-}(\cat{P})\) be the canonical embedding. \par
	For a complex \(X^\bullet\), we let \(\tau_{k}X^\bullet\) be the complex obtained from \(X^\bullet\) by setting \(\tau_{k}X^n = X^n\) for \(n \geq k\), and \(\tau_{k}X^n = 0\) otherwise.
	
	\begin{lemma}\label{lemma-truncation}
		Let \(M\) be an object in \(\cat{E}\) and \(Q^\bullet\) be a complex in \(\hocat{b}(\cat{Q}_0)\). Any morphism \(f^\bullet : \iota(M) \rightarrow Q^\bullet\) in \(\hocat{-}(\cat{P})\) factors through the canonical map \(\tau_0Q^\bullet \rightarrow Q^\bullet\).
	\end{lemma}
	
	\begin{proof}
		Let \(k\) be the smallest integer such that \(Q^k \neq 0\). If \(k \geq 0\), the statement is trivial, so we assume \(k < 0\). By induction, it will suffice to show that \(f^\bullet\) factors through \(\tau_{k+1} Q^\bullet \rightarrow Q^\bullet\). \par
		We write \(P^\bullet \coloneq \iota (M)\). Then locally around \(k\), we have the commutative solid diagram
		\begin{center}
		\begin{tikzcd}
			&&& C && \\
			\cdots & {P^{k-1}} & {P^k} && {P^{k+1}} & \cdots \\
			\cdots & 0 & {Q^k} && {Q^{k+1}} & \cdots
			\arrow["y", tail, from=1-4, to=2-5]
			\arrow[from=2-1, to=2-2]
			\arrow["{d_P^{k-1}}", from=2-2, to=2-3]
			\arrow[from=2-2, to=3-2]
			\arrow["x", two heads, from=2-3, to=1-4]
			\arrow["{d_P^k}", from=2-3, to=2-5]
			\arrow["{f^k}", from=2-3, to=3-3]
			\arrow[from=2-5, to=2-6]
			\arrow["h", dashed, from=2-5, to=3-3]
			\arrow["{f^{k+1}}", from=2-5, to=3-5]
			\arrow[from=3-1, to=3-2]
			\arrow[from=3-2, to=3-3]
			\arrow["{d_Q^k}"', from=3-3, to=3-5]
			\arrow[from=3-5, to=3-6]
		\end{tikzcd}
		\end{center}
		with \(x\) a deflation which is the cokernel of \(d_P^{k-1}\), and \(y\) an inflation. Since \(f^k \circ d_P^{k-1} = 0\), the morphism \(f^k\) factors through \(x\). Since \(y\) is an inflation and \(Q^k\) an injective object, it factors further through \(y\). So we obtain the dashed morphism \(h:P^{k+1}\rightarrow Q^k\) such that \(h\circ d_P^k = f^k\). So \(f^\bullet\) is homotopic to the morphism \(\tilde{f}^\bullet : P^\bullet \rightarrow Q^\bullet\), which is given by
		\begin{align*}
			\tilde{f}^m = \begin{cases}
				0 & m \leq k \\
				f^{k+1}-d_Q^k\circ h & m = k+1 \\
				f^m & m \geq k+2
			\end{cases}.
		\end{align*}
		Clearly \(\tilde{f}^\bullet\) factors through \(\tau_{k+1}Q^\bullet \rightarrow Q^\bullet\), which concludes the proof.
	\end{proof}
	
	\begin{lemma}\label{lemma-kernel}
		The kernel of the functor \(\Phi : \cat{E} \rightarrow \hocat{-}(\cat{P})/\hocat{b}(\cat{Q}_0)\) is exactly \([\cat{Q}_0]\). 
	\end{lemma}
	
	\begin{proof}
		Clearly \([\cat{Q}_0] \subseteq \ker \Phi\). \par
		Conversely, let \(f : M \rightarrow N\) be a morphism in \(\cat{E}\) such that \(\Phi(f) = 0 \). This means that \(\iota(f)\) factors as
		\begin{align*}
			\iota(M) \rightarrow Q^\bullet \rightarrow \iota(N)
		\end{align*}
		for some \(Q^\bullet \in \hocat{b}(\cat{Q}_0)\). Now \Cref{lemma-truncation} implies that \(\iota(f)\) factors as
		\begin{align*}
			\iota(M) \rightarrow \tau_0 Q^\bullet \rightarrow \iota(N),
		\end{align*}
		which in turn implies that \(f\) factors through \(Q^0\). This proves the claim.
	\end{proof}
	
	\begin{lemma}\label{lemma-full}
		The functor \(\Phi : \cat{E} \rightarrow \hocat{-}(\cat{P})/\hocat{b}(\cat{Q}_0)\) is full.
	\end{lemma}
	
	\begin{proof}
		Let \(f^\bullet : \Phi(M) \rightarrow \Phi(N)\) be a morphism in \(\hocat{-}(\cat{P})/\hocat{b}(\cat{Q}_0)\). We can represent \(f^\bullet\) by a roof
		\begin{center}
		\begin{tikzcd}
			&& {P^\bullet} & \\
			& {\iota(M)} && {\iota(N)} \\
			{Q^\bullet}
			\arrow["q", from=1-3, to=2-2]
			\arrow["p"', from=1-3, to=2-4]
			\arrow[from=2-2, to=3-1]
		\end{tikzcd}.
		\end{center}
		Using \Cref{lemma-truncation}, we obtain the following diagram
		\begin{center}
		\begin{tikzcd}
			{\Sigma^{-1}\tau_0Q^\bullet} & {\tilde{P}^\bullet} & {\iota(A)} & {\tau_0Q^\bullet} \\
			{\Sigma^{-1}Q^\bullet} & {P^\bullet} & {\iota(A)} & {Q^\bullet}
			\arrow[from=1-1, to=1-2]
			\arrow[from=1-1, to=2-1]
			\arrow["{\tilde{q}}", from=1-2, to=1-3]
			\arrow["h", dashed, from=1-2, to=2-2]
			\arrow[from=1-3, to=1-4]
			\arrow[equals, from=1-3, to=2-3]
			\arrow[from=1-4, to=2-4]
			\arrow[from=2-1, to=2-2]
			\arrow["q", from=2-2, to=2-3]
			\arrow[from=2-3, to=2-4]
		\end{tikzcd},
		\end{center}
		where the rows are triangles in \(\hocat{-}(\cat{P})\). So the morphism \(f\) is also represented by the roof
		\begin{center}
		\begin{tikzcd}
			&& {\tilde{P}^\bullet} & \\
			& {\iota(M)} && {\iota(N)} \\
			{\tau_0 Q^\bullet}
			\arrow["\tilde{q}", from=1-3, to=2-2]
			\arrow["p\circ h"', from=1-3, to=2-4]
			\arrow[from=2-2, to=3-1]
		\end{tikzcd}.
		\end{center}
		As \(\Hom{\hocat{-}(\cat{P})}(\Sigma^{-1}\tau_0Q^\bullet,\iota(N))=0\), the morphism \(p \circ h\) factors through \(\tilde{q}\). Explicitly, the functor \(\iota\) being full, there exists a morphism \(\tilde{f} : M \rightarrow N\) which makes the following diagram in \(\hocat{-}(\cat{P})\) commute
		\begin{center}
		\begin{tikzcd}
			& {\tilde{P}^\bullet} & \\
			{\iota(M)} & {\iota(M)} & {\iota(N)}
			\arrow["{\tilde{q}}"', from=1-2, to=2-1]
			\arrow["{\tilde{q}}", from=1-2, to=2-2]
			\arrow["{p\circ h}", from=1-2, to=2-3]
			\arrow[equals, from=2-2, to=2-1]
			\arrow["{\Phi(\tilde{f}})"', from=2-2, to=2-3]
		\end{tikzcd}.
		\end{center}
		In particular, \(\Phi(\tilde{f}) = f\) in \(\hocat{-}(\cat{P})/\hocat{b}(\cat{Q}_0)\).
	\end{proof}
	
	\begin{remark}\label{remark-full}
		We notice that we can replace \(\Phi(N)\) in the proof by any complex \(X^\bullet\) satisfying \(X^k = 0\) for \(k > 0\). In particular, the natural homomorphism
		\begin{align*}
			\Hom{\hocat{-}(\cat{P})}(\iota(M),X^\bullet) \rightarrow \Hom{\hocat{-}(\cat{P})/\hocat{b}(\cat{Q}_0)}(\Phi(M),X^\bullet)
		\end{align*}
		is surjective. 
	\end{remark}
	
	\begin{proposition}\label{proposition-embedding}
		The functor \(\Phi : \cat{E} \rightarrow \hocat{-}(\cat{P})/\hocat{b}(\cat{Q}_0)\) induces a full and faithful functor \(\Psi : \cat{C} \rightarrow \hocat{-}(\cat{P})/\hocat{b}(\cat{Q}_0)\). The essential image of \(\Psi\) is extension closed, and it induces an equivalence of extriangulated categories between \(\cat{C}\) and its essential image. 
	\end{proposition}
	
	\begin{proof}
		To simplify notation, we will write \(\cat{D} \coloneq \hocat{-}(\cat{P})/\hocat{b}(\cat{Q}_0)\). \par
		Combining \Cref{lemma-kernel} and \Cref{lemma-full} gives that the functor \(\Phi : \cat{E} \rightarrow \cat{D}\) induces a full and faithful extriangulated functor \(\Psi : \cat{C} \rightarrow \cat{D}\). According to \Cref{remark-full}, the natural homomorphism
		\begin{align}\label{eq-surj}
			\Hom{\hocat{-}(\cat{P})}(\iota(M),\Sigma\iota(N)) \rightarrow \Hom{\cat{D}}(\Phi(M),\Sigma\Phi(N))
		\end{align}
		is surjective. This means that every triangle
		\begin{align*}
			\Phi(M) \rightarrow X \rightarrow \Phi(N) \rightarrow \Sigma\Phi(M)
		\end{align*}
		in \(\cat{D}\) is the image under the quotient functor of a triangle
		\begin{align*}
			\iota(M) \rightarrow \tilde{X} \rightarrow \iota(N) \rightarrow \Sigma\iota(M)
		\end{align*}
		in \(\hocat{-}(\cat{P})\). In particular, the essential image of \(\Phi\), and hence of \(\Psi\), is extension closed. \par
		To conclude the proof, we need to show that the natural homomorphism 
		\begin{align*}
			\Ext{\cat{C}}(M,N) \rightarrow \Hom{\cat{D}}(\Psi(M),\Sigma\Psi(N))
		\end{align*}
		is an isomorphism. Let	
		\begin{align}\label{eq-conflation}
			L \inflation P \deflation M
		\end{align}
		be a conflation in \(\cat{C}\), with \(P\) a projective object. Applying \(\Hom{\cat{C}}(-,N)\) to this conflation, we obtain the exact sequence
		\begin{align}\label{eq-exact1}
			\Hom{\cat{C}}(P,N) \rightarrow \Hom{\cat{C}}(L,N) \deflation \Ext{\cat{C}}(M,N).
		\end{align}
		Similarly, applying the functor \(\Psi\) to \eqref{eq-conflation} we get the triangle
		\begin{align*}
			\Psi(L) \rightarrow \Psi(P) \rightarrow \Psi(M) \rightarrow \Sigma\Psi(L).
		\end{align*}
		Using \eqref{eq-surj}, we deduce \(\Hom{\cat{D}}(\Sigma^{-1}\Psi(P),\Psi(N)) = 0 \). Therefore applying \(\Hom{\cat{D}}(-,\Psi(N))\) to the above triangle gives the exact sequence
		\begin{align}\label{eq-exact2}
			\Hom{\cat{D}}(\Psi(P),\Psi(N)) \rightarrow \Hom{\cat{D}}(\Psi(L),\Psi(N)) \deflation \Hom{\cat{D}}(\Sigma^{-1}\Psi(M),\Psi(N)).
		\end{align}
		Comparing \eqref{eq-exact1} and \eqref{eq-exact2}, and using the full faithfulness of \(\Psi\), we obtain the isomorphism
		\begin{align*}
			\Ext{\cat{C}}(M,N) \simeq \Hom{\cat{D}}(\Sigma^{-1}\Psi(M),\Psi(N)) \simeq \Hom{\cat{D}}(\Psi(M),\Sigma\Psi(N)),
		\end{align*}
		which concludes our proof.
	\end{proof}

	\section{The \(d\)-Auslander case}\label{d-Auslander}
	
	We will use \Cref{proposition-embedding} to characterize reduced algebraic extriangulated \(d\)-Auslander categories. Throughout this section, \(\cat{E}\) is an exact \(d\)-Auslander category, \(\cat{P}\) is the subcategory of projective objects, and \(\cat{Q}\) the subcategory of projective-injective objects.
	
	\begin{definition}
		Let \(k \in \mathbb{N}\). An object \(L\) in \(\cat{E}\) is said to have injective dominant dimension at least \(k\) if there exists an injective resolution of \(L\) whose first \(k\) terms are projective. In this case, we write \(\idd{L} \geq k\).
	\end{definition}
	
	\begin{remark}
		An object \(L\) in \(\cat{E}\) is projective-injective if and only if \(\idd{L}\geq k\) for any \(k \in \mathbb{N}\).
	\end{remark}
	
	\begin{lemma}\label{lemma-common-injective}
		Let \(f: A\inflation B\) and \(g: A\inflation C\) be inflations. Let \(i : B \inflation I_B\) and \(j : C \inflation I_C\) be inflations with injective codomains. Then there exist morphisms \(r : B \rightarrow I_C\) and \(s : C \rightarrow I_B\) making the following square commute
		\begin{center}
		\begin{tikzcd}
			A & C \\
			B & {I_B\oplus I_C}
			\arrow["g", tail, from=1-1, to=1-2]
			\arrow["f"', tail, from=1-1, to=2-1]
			\arrow["{\mat{s\\j}}", from=1-2, to=2-2]
			\arrow["{\mat{i\\r}}"', from=2-1, to=2-2]
		\end{tikzcd}.
		\end{center}
	\end{lemma}
	
	\begin{proof}
		As the objects \(I_B\) and \(I_C\) are injective, there exist morphisms \(r:B \rightarrow I_C\) and \(s : C \rightarrow I_B\) making the following diagrams commute
		\begin{center}
			\begin{tikzcd}
				A & C \\
				B \\
				{I_B}
				\arrow["g", tail, from=1-1, to=1-2]
				\arrow["f"', tail, from=1-1, to=2-1]
				\arrow["s", from=1-2, to=3-1]
				\arrow["i"', tail, from=2-1, to=3-1]
			\end{tikzcd}
			\qquad\qquad\qquad
			\begin{tikzcd}
				A & B \\
				C \\
				{I_C}
				\arrow["f", tail, from=1-1, to=1-2]
				\arrow["g"', tail, from=1-1, to=2-1]
				\arrow["r", from=1-2, to=3-1]
				\arrow["j"', tail, from=2-1, to=3-1]
			\end{tikzcd}.
		\end{center}
		It is straightforward to verify that all these morphisms fit into the desired square, which commutes.
	\end{proof}
	
	\begin{lemma}\label{lemma-idd}
		Suppose \(L\) fits in a conflation
		\begin{align*}
			L \inflation Q \deflation M
		\end{align*}
		with \(Q\) a projective-injective object. Then \(\idd{L} \geq k\) if and only if \(\idd{M} \geq k-1\).
	\end{lemma}
	
	\begin{proof}
		If \(\idd{M} \geq k-1\), then obviously \(\idd{L} \geq k\). \par
		Conversely, we assume \(k > 1\), otherwise there is nothing to prove. So if \(\idd{L}\geq k\), by definition there exists a projective-injective object \(\tilde{Q}\) and a conflation \(L \inflation \tilde{Q} \deflation \tilde{M}\) such that \(\idd{\tilde{M}} \geq k -1\). By \Cref{lemma-common-injective}, we have the following commutative diagram
		\begin{center}
		\begin{tikzcd}
			L & Q \\
			{\tilde{Q}} & {Q\oplus\tilde{Q}}
			\arrow[tail, from=1-1, to=1-2]
			\arrow[tail, from=1-1, to=2-1]
			\arrow[tail, from=1-2, to=2-2]
			\arrow[tail, from=2-1, to=2-2]
		\end{tikzcd}.
		\end{center}
		We construct the following pullback diagram
		\begin{center}
		\begin{tikzcd}
			L & Q & M \\
			{\tilde{Q}} & E & M \\
			{\tilde{M}} & {\tilde{M}}
			\arrow[tail, from=1-1, to=1-2]
			\arrow[tail, from=1-1, to=2-1]
			\arrow[two heads, from=1-2, to=1-3]
			\arrow[tail, from=1-2, to=2-2]
			\arrow[equals, from=1-3, to=2-3]
			\arrow[tail, from=2-1, to=2-2]
			\arrow[two heads, from=2-1, to=3-1]
			\arrow[two heads, from=2-2, to=2-3]
			\arrow[two heads, from=2-2, to=3-2]
			\arrow[equals, from=3-1, to=3-2]
		\end{tikzcd}.
		\end{center}
		Since \(Q\) and \(\tilde{Q}\) are projective-injective objects, we deduce an isomorphism \(M\oplus \tilde{Q} \simeq \tilde{M} \oplus Q\), both objects being cokernels of the morphism \(L\inflation Q\oplus\tilde{Q}\). It follows that \(\idd{M} = \idd{\tilde{M}} \geq k-1\), which concludes the proof.
	\end{proof}
	
	\begin{proposition}\label{proposition-idd}
		Let \(L\) be an object such that \(k \coloneq \pd{L} < d+1\). Then there is a conflation
		\begin{align*}
			L \inflation M \deflation N
		\end{align*}
		with \(\idd{M} \geq d-k\) and \(\idd{N} \geq d-1\).
	\end{proposition}
	
	\begin{proof}
		We prove the proposition by induction on \(k\). If \(k = 0\), then \(L\) is projective and the conflation \(L \inflation L \deflation 0\) has the desired properties. \par
		If \(k>0\), then there exists a conflation of the form \(L' \inflation P \deflation L\) with \(P\) a projective object. By induction, \(L'\) fits in a conflation of the form \(L' \inflation M' \deflation N'\) with \(\idd{M'} \geq d - k + 1\) and \( \idd{N'} \geq d-1\). Notice that \(\idd{M'} > 0\) by hypothesis. \par
		According to \Cref{lemma-common-injective}, there exists a projective-injective object \(Q\) which fits in a commutative diagram
		\begin{center}
		\begin{tikzcd}
			{L'} & P \\
			{M'} & Q
			\arrow[tail, from=1-1, to=1-2]
			\arrow[tail, from=1-1, to=2-1]
			\arrow[tail, from=1-2, to=2-2]
			\arrow[tail, from=2-1, to=2-2]
		\end{tikzcd}.
		\end{center}
		We consider the following pullback diagrams
		\begin{center}
		\begin{tikzcd}
			{L'} & P & L \\
			{L'} & Q & M \\
			& N & N
			\arrow[tail, from=1-1, to=1-2]
			\arrow[equals, from=1-1, to=2-1]
			\arrow[two heads, from=1-2, to=1-3]
			\arrow[tail, from=1-2, to=2-2]
			\arrow[tail, from=1-3, to=2-3]
			\arrow[tail, from=2-1, to=2-2]
			\arrow[two heads, from=2-2, to=2-3]
			\arrow[two heads, from=2-2, to=3-2]
			\arrow[two heads, from=2-3, to=3-3]
			\arrow[equals, from=3-2, to=3-3]
		\end{tikzcd}
		\qquad\qquad\qquad
		\begin{tikzcd}
			{L'} & {M'} & {N'} \\
			{L'} & Q & M \\
			& E & E
			\arrow[tail, from=1-1, to=1-2]
			\arrow[equals, from=1-1, to=2-1]
			\arrow[two heads, from=1-2, to=1-3]
			\arrow[tail, from=1-2, to=2-2]
			\arrow[tail, from=1-3, to=2-3]
			\arrow[tail, from=2-1, to=2-2]
			\arrow[two heads, from=2-2, to=2-3]
			\arrow[two heads, from=2-2, to=3-2]
			\arrow[two heads, from=2-3, to=3-3]
			\arrow[equals, from=3-2, to=3-3]
		\end{tikzcd}.
		\end{center}
		By \Cref{lemma-idd}, we have \(\idd{E} \geq d-k\) and \(\idd{N} \geq d-1\). From the horseshoe lemma, we deduce that \(\idd{M} \geq d-k\). Thus the conflation \(L \inflation M \deflation N\) has the desired properties. 
	\end{proof}
	
	An immediate consequence of \Cref{proposition-idd} is the following corollary.
	
	\begin{corollary}\label{corollary-pi-envelope}
		Let \(L\) be such that \(\pd{L} < d+1\). Then there exists an inflation \(L \inflation Q\) where \(Q\) is projective-injective.  \qed
	\end{corollary}
	
	Let us write \(\cat{D} \coloneq \hocat{-}(\cat{P})/\hocat{b}(\cat{Q})\), and let \(\cat{D}^{[-d-1,0]}\) be the essential image of \(\hocat{[-d-1,0]}(\cat{P})\) under the quotient functor \(\hocat{-}(\cat{P}) \rightarrow \cat{D}\). We can now prove the main theorem of this section.
	
	\begin{theorem}\label{theorem-ess-image}
		The essential image of the functor \(\Phi : \cat{E} \rightarrow \cat{D}\) is \(\cat{D}^{[-d-1,0]}\).
	\end{theorem}
	
	\begin{proof}
		Since \(\cat{E}\) has global dimension at most \(d+1\), the image of \(\Phi\) lies in \(\cat{D}^{[-d-1,0]}\). \par
		Now let \(P^\bullet\) be an object in \(\cat{D}^{[-d-1,0]}\), which is represented by the complex
		\begin{align*}
			\cdots \rightarrow 0 \rightarrow P^{-d-1} \xrightarrow{f^{-d-1}} P^{-d} \xrightarrow{f^{-d}} \cdots \xrightarrow{f^{-1}} P^0 \rightarrow 0 \rightarrow \cdots.
		\end{align*}
		We would like to show that \(P^\bullet \simeq \Phi(M)\) for some \(M \in \cat{E}\). Our strategy is to construct inductively a sequence of complexes \(P^\bullet_{d+1},...,P^\bullet_0\) in \(\cat{D}^{[-d-1,0]}\) such that
		\begin{enumerate}
			\item \(P^\bullet_{d+1} \coloneq P^\bullet\),
			\item for all \(1 \leq k \leq d+1\) there is an isomorphism \(P^\bullet_k \simeq P^\bullet_{k-1}\) in \(\cat{D}\),
			\item \(P^\bullet_k\) is exact up to degree \(k\), which means that for any \(k \leq l \leq d+1\) there are conflations 
			\begin{center}
			\begin{tikzcd}
				{C_k^{-l-1}} & {P_k^{-l}} & {C_k^{-l}}
				\arrow["{a_k^{-l-1}}", tail, from=1-1, to=1-2]
				\arrow["{b_k^{-l}}", two heads, from=1-2, to=1-3]
			\end{tikzcd}
			\end{center}
			in \(\cat{E}\) such that \(f^{-l}_k = a_{k}^{-l} \circ b_k^{-l} \) for \(k+1 \leq l \leq d+1\).
		\end{enumerate}
		If such a sequence exists, then the third property implies that \(P^\bullet_0\) is in the essential image of \(\Phi\). \par
		Certainly \(P^\bullet_{d+1} \coloneq P^\bullet\) is exact up to degree \(d+1\) (which is an empty condition). Suppose now we have constructed \(P^\bullet_k\) for some \(k > 0\). Then by assumption we have the following solid diagram
		\begin{center}
		\begin{tikzcd}[column sep=scriptsize]
			{P_k^{-k-1}} && {P_k^{-k}} && {P_k^{-k+1}} \\
			& {C_k^{-k-1}} && {C_k^{-k}}
			\arrow["{f_k^{-k-1}}", from=1-1, to=1-3]
			\arrow["{b_k^{-k-1}}"', two heads, from=1-1, to=2-2]
			\arrow["{f_k^{-k}}", from=1-3, to=1-5]
			\arrow["{b_k^{-k}}"', two heads, from=1-3, to=2-4]
			\arrow["{a_k^{-k-1}}"', tail, from=2-2, to=1-3]
			\arrow["{h_1}"', dashed, from=2-4, to=1-5]
		\end{tikzcd}
		\end{center}
		which commutes in \(\cat{E}\). As \(f^{-k}_k\circ a^{-k-1}_k \circ b^{-k-1}_k = f^{-k}_k\circ f^{-k-1}_k = 0\) and \(b^{-k-1}_k\) is an epimorphism, we obtain the dashed morphism \(h_1 : C^{-k}_k \rightarrow P^{-k+1}_k\) such that \(f^{-k}_k = h_1 \circ b_{k}^{-k}\). \par
		But \(\pd{C_k^{-k}} \leq d+1-k < d+1\), so \Cref{corollary-pi-envelope} implies there exists an inflation \(h_2 :C^{-k}_k \inflation Q\), with \(Q\) a projective-injective object. We then define \(P^\bullet_{k-1}\) as the complex
		\begin{align*}
			\cdots \xrightarrow{f^{-k-1}_k} P^{-k}_k \xrightarrow{\mat{h_1 \\ h_2}\circ b_k^{-k}} P_k^{-k+1} \oplus Q \xrightarrow{\mat{f^{-k+1}_k & 0}} P_k^{-k+2} \xrightarrow{f_k^{-k+2}} \cdots \xrightarrow{f^{-1}_k} P_k^{0} \rightarrow 0 \rightarrow \cdots.
		\end{align*}
		As \(P^\bullet_{k-1} \simeq \Sigma^{-1} \cone{P^\bullet_{k} \rightarrow \Sigma^{k}Q}\) in \(\hocat{-}(\cat{P})\), we have \(P^\bullet_{k-1} \simeq P^\bullet_k\) in \(\cat{D}\). Because \(\mat{h_1 \\ h_2}\) is an inflation in \(\cat{E}\), the complex \(P^\bullet_{k-1}\) is exact up to degree \(k-1\). \par
		Repeating this argument, we obtain the desired sequence of complexes with isomorphisms
		\begin{align*}
			P^\bullet_{d+1} \simeq P^\bullet_d \simeq \cdots \simeq P^\bullet_0
		\end{align*}
		in \(\cat{D}\). As \(P^\bullet_0\) is in the essential image of \(\Phi\), so is \(P^\bullet = P^\bullet_{d+1}\), which completes the proof. 
	\end{proof}
	
	\begin{corollary}\label{corollary-reduced-d-Auslander}
		The functor \(\Psi : \cat{E}/[\cat{Q}] \rightarrow \cat{D}\) of \Cref{proposition-embedding} induces an equivalence of extriangulated categories \(\Psi : \cat{E}/[\cat{Q}] \rightarrow \cat{D}^{[-d-1,0]}\). \qed
	\end{corollary}
	
	\section{Truncated homotopy categories}\label{Main}
	
	Extriangulated \(d\)-Auslander categories behave in many ways like truncated homotopy categories, that is, categories of the form \(\hocat{[-d-1,0]}(\cat{A})\) for some additive category \(\cat{A}\). The purpose of this section is to give a complete description of when an algebraic extriangulated \(d\)-Auslander category \(\cat{C}\) admits an extriangulated ideal quotient which is equivalent to a category \(\hocat{[-d-1,0]}(\cat{A})\). \par
	We assume \(\cat{E}\) is an exact \(d\)-Auslander category. We denote by \(\cat{P}\) its subcategory of projective objects, by \(\cat{I}\) its subcategory of injective objects, and by \(\cat{Q}\) its subcategory of projective-injective objects. 
	
	\begin{lemma}\label{lemma-hom-to-projective}
		Suppose \(\Extg{\cat{E}}^k(I,P) = 0\) for any \(1\leq k \leq d\), \(I \in \cat{I}\) and \(P \in \cat{P}\). Let \(M\) be an object in \(\cat{E}\) such that \(\id{M} < d+1\). Then for any projective \(P'\), we have
		\begin{align*}
			\Hom{\cat{E}}(M,P') \subseteq 
			\begin{cases}
				[\cat{Q}](M,P') & \text{if}\quad \pd{M} < d+1, \\
				[\cat{I}](M,P') & \text{if}\quad \pd{M} = d+1.
			\end{cases}
		\end{align*}
	\end{lemma}
	
	\begin{proof}
		We begin by noting that it is enough to prove the statement for the reduced extriangulated category \(\cat{E}/[\cat{Q}]\), which, according to \Cref{corollary-reduced-d-Auslander}, is equivalent as an extriangulated category, under the functor \(\Psi\), to \(\cat{D}^{[-d-1,0]}\). \par
		We see that an object \(I \in \cat{E}\) is injective if and only if it has the form \(\Psi(I) \simeq \Sigma^{d+1}\Psi(P)\) for some projective \(P \in \cat{P}\). In particular, the assumption of the lemma implies the vanishing condition
		\begin{align}\label{eq-vanishing}
			\Hom{\cat{D}^{[-d-1,0]}}(\Sigma^k\Psi(P),\Psi(P')) = 0
		\end{align}
		for all \(P,P' \in \cat{P}\) and \(1 \leq k \leq d\). Our goal is thus to prove that for \(M \in \cat{E}\) with \(\id{M} <d+1\), we have
		\begin{align*}
			\Hom{\cat{D}^{[-d-1,0]}}(\Psi(M),\Psi(P')) \subseteq 
			\begin{cases}
				0 & \text{if}\quad \pd{M} < d+1, \\
				[\Psi(\cat{I})](\Psi(M),\Psi(P')) & \text{if}\quad \pd{M} = d+1,
			\end{cases}
		\end{align*}
		for any projective object \(P'\) in \(\cat{E}\). If \(d = 0\) the statement is trivial, so we assume \(d>0\). \par
		We prove the claim by induction on the projective dimension of \(M\). If \(\pd{M} = 0\), then either \(M\) is projective-injective and the statement holds trivially, or \(\id{M} = d+1\), which means \(M\) does not satisfy the condition of the lemma. \par
		If \(\pd{M} = 1\), the condition \(\id{M} < d+1\) forces \(\Psi(M) \simeq \Sigma \Psi(P)\) for some \(P \in \cat{P}\), and
		\begin{align*}
			\Hom{\cat{D}^{[-d-1,0]}}(\Psi(M),\Psi(P')) = 0
		\end{align*}
		for any \(P' \in \cat{P}\) follows from \eqref{eq-vanishing}. \par
		Now set \(l = \pd{M}\), and suppose that \(l > 1\). Then the complex \(\Psi(M)\) is isomorphic in \(\cat{D}^{[-d-1,0]}\) to a complex \(\tilde{P}^\bullet\) with \(\tilde{P}^k = 0\) for \(k < -l\), and moreover \(\tilde{P}^{-l} \neq 0\). The complex \(\tilde{P}^\bullet\) fits in a conflation \(X^\bullet \inflation \tilde{P}^\bullet \deflation \Sigma^{l}\tilde{P}^{-l}\), where \(X^\bullet = \tau_{-l+1}\tilde{P}^\bullet\). It is not hard to verify that \(\id{X^\bullet} \leq \id{\tilde{P}^\bullet}\) and \(\pd{X^\bullet} < \pd{\tilde{P}^\bullet}\). Applying the functor \(\Hom{\cat{D}^{[-d-1,0]}}(-,\Psi(P'))\) to this conflation and using the isomorphism \(\Psi(M) \simeq \tilde{P}^\bullet\), we obtain the exact sequence
		\begin{align*}
			\Hom{\cat{D}^{[-d-1,0]}}(\Sigma^{l}\tilde{P}^{-l},\Psi(P'))  \rightarrow \Hom{\cat{D}^{[-d-1,0]}}(\Psi(M),\Psi(P')) \rightarrow \Hom{\cat{D}^{[-d-1,0]}}(X^\bullet,\Psi(P')).
		\end{align*}
		By induction hypothesis, \(\Hom{\cat{D}^{[-d-1,0]}}(X^\bullet,\Psi(P')) = 0\), so that
		\begin{align*}
			\Hom{\cat{D}^{[-d-1,0]}}(\Sigma^{l}\tilde{P}^{-l},\Psi(P'))  \deflation \Hom{\cat{D}^{[-d-1,0]}}(\Psi(M),\Psi(P')).
		\end{align*}
		If \(l < d+1\), we have \(\Hom{\cat{D}^{[-d-1,0]}}(\Sigma^{l}\tilde{P}^{-l},\Psi(P')) = 0\) by \eqref{eq-vanishing}, and so
		\begin{align*}
			\Hom{\cat{D}^{[-d-1,0]}}(\Psi(M),\Psi(P')) = 0.
		\end{align*}  
		Otherwise, \(\Sigma^{l}\tilde{P}^{-l}\) is injective, so
		\begin{align*}
			\Hom{\cat{D}^{[-d-1,0]}}(\Psi(M),\Psi(P')) \subseteq [\cat{I}](\Psi(M),\Psi(P')).
		\end{align*}
		This concludes the proof.
	\end{proof}
	
	\begin{lemma}\label{lemma-Q-factor}
		Suppose \(\Extg{\cat{E}}^k(I,P) = 0\) for any \(1\leq k \leq d\), \(I \in \cat{I}\) and \(P \in \cat{P}\). \par 
		Let \(f: P \deflation C\) be a deflation in \(\cat{E}\) with \(P\) projective and \(\pd{C} < d+1\). If \(g:C \rightarrow P'\) is a morphism with \(P'\) projective such that \(g\circ f \in [\cat{Q}]\), then \(g \in [\cat{Q}]\). 
	\end{lemma}
	
	\begin{proof}
		Let \(a : P \rightarrow Q\) and \(b : Q \rightarrow P'\) be a factorization of \(g\circ f\) through a projective-injective object \(Q\). Without loss of generality, we assume \(a\) is an inflation. \par
		We construct the following push-out diagram
		\begin{center}
		\begin{tikzcd}
			P & C \\
			Q & E
			\arrow["f", two heads, from=1-1, to=1-2]
			\arrow["a"', tail, from=1-1, to=2-1]
			\arrow["x", tail, from=1-2, to=2-2]
			\arrow[two heads, from=2-1, to=2-2]
		\end{tikzcd},
		\end{center}
		where \(x : C \inflation E\) is the push-out of \(a\) along \(f\). Since \(b\circ a = g \circ f\), there exists a morphism \(y : E \rightarrow P'\) such that \(y \circ x = g\). Since \(\id{E} < d+1\) and \(\pd{E} \leq \pd{C} < d+1\), \Cref{lemma-hom-to-projective} implies that \(y \in [\cat{Q}]\). Therefore \(g = y \circ x \in [\cat{Q}]\), which concludes the proof.
	\end{proof}
	
	\begin{lemma}\label{lemma-only-zero}
		Let \(P^\bullet\) and \(\tilde{P}^\bullet\) be complexes in \(\hocat{b}(\cat{P})\) which are projective resolutions of objects in \(\cat{E}\). Assume \(P^0 \neq 0\). Let \(f^\bullet : P^\bullet \rightarrow \tilde{P}^\bullet\) be a morphism such that \(f^i \in [\cat{Q}]\) for all \(i\). Then \(f^\bullet\) is homotopic to a morphism \(\tilde{f}^\bullet\) satisfying \(\tilde{f}^i = 0\) for all \(i \neq 0\) and \(\tilde{f}^0 \in [\cat{Q}]\).
	\end{lemma}
	
	\begin{proof}
		Let \(n \in \mathbb{Z}\) be the minimal integer such that \(f^n \neq 0\). We will show that if \(n < 0\), then \(f^\bullet\) is homotopic to a morphism \(\bar{f}^\bullet\) such that \(\bar{f}^k \in [\cat{Q}]\) for all \(k \in \mathbb{Z}\) and \(\bar{f}^{k} = 0\) for all \(k \leq n\). The claim of the proof then follows by recursion. \par
		Let us denote by \(C^{k-1}\xinflation{a^k} P^k \xdeflation{b^k} C^k\) the conflations that make up the complex \(P^\bullet\). (The complex \(P^\bullet\) is nothing but a projective resolution of \(C^0\).) Moreover, let \(g^k : P^k \rightarrow Q^k\) and \(h^k : Q^k \rightarrow \tilde{P}^k\) be such that \(f^k = h^k \circ g^k\) and \(Q^k\) is projective-injective. We can assume \(g^{k}\) are inflations. Then around \(n\), the morphism \(f^\bullet\) corresponds to solid commutative diagram 
		\begin{center}
		\begin{tikzcd}
			\cdots & {P^{n-1}} & {C^{n-1}} & {P^n} & {C^n} & {P^{n+1}} & \cdots \\
			&&& {Q^n} && {Q^{n+1}} \\
			\cdots & {\tilde{P}^{n-1}} && {\tilde{P}^n} && {\tilde{P}^{n+1}} & \cdots
			\arrow[from=1-1, to=1-2]
			\arrow[two heads, from=1-2, to=1-3]
			\arrow["0", from=1-2, to=3-2]
			\arrow["{a^n}", tail, from=1-3, to=1-4]
			\arrow["{b^n}", two heads, from=1-4, to=1-5]
			\arrow["{g^n}", tail, from=1-4, to=2-4]
			\arrow["{a^{n+1}}", tail, from=1-5, to=1-6]
			\arrow["j"', bend left = 30, dashed, from=1-5, to=3-4]
			\arrow[from=1-6, to=1-7]
			\arrow["{g^{n+1}}", tail, from=1-6, to=2-6]
			\arrow["{h^n}", from=2-4, to=3-4]
			\arrow["{h^{n+1}}", from=2-6, to=3-6]
			\arrow[from=3-1, to=3-2]
			\arrow["\tilde{d}^{n-1}", from=3-2, to=3-4]
			\arrow["\tilde{d}^{n}", from=3-4, to=3-6]
			\arrow[from=3-6, to=3-7]
		\end{tikzcd}.
		\end{center}
		Since \(h^n \circ g^n \circ a^n = 0\), there exists a morphism \(j : C^n \rightarrow P\) such that \(j\circ b^n = h^n \circ g^n\). We notice that \(j\circ b^n \in [\cat{Q}]\) and \(\pd{C^n} < d+1\). By \Cref{lemma-Q-factor}, the morphism \(j\) factors through a projective injective object \(Q\). Let \(j_1 : C^n\rightarrow Q\) and \(j_2 : Q \rightarrow \tilde{P}^n\) be such a factorization. Since \(a^{n+1}\) is an inflation, there exists a morphism \(\phi : P^{n+1} \rightarrow Q\) such that \(j_1 = \phi \circ a^{n+1}\). \par
	Finally, we define the morphism \(\bar{f}^\bullet\) by 
	\begin{align*}
		\bar{f}^k = 
		\begin{cases}
			0 & \text{if}\quad k = n, \\
			f^k - \tilde{d}^n \circ j_2 \circ \phi & \text{if}\quad k = n+1, \\
			f^k & \text{otherwise}.
		\end{cases}
	\end{align*}
	It is immediate to verify that \(f^\bullet\) is homotopic to \(\bar{f}^\bullet\).
	\end{proof}
	
	The additive quotient \(\cat{P} \rightarrow \cat{P}/[\cat{Q}]\) induces a canonical triangulated functor \(\hocat{-}(\cat{P})\rightarrow\hocat{-}(\cat{P}/[\cat{Q}])\) which vanishes on \(\hocat{b}(\cat{Q})\), thus inducing a triangulated functor \(\rho : \hocat{-}(\cat{P})/\hocat{b}(\cat{Q}) \rightarrow \hocat{-}(\cat{P}/[\cat{Q}])\). \par
	We further notice that the image of \(\cat{D}^{[-d-1,0]}\) under \(\rho\) lies in \(\hocat{[-d-1,0]}(\cat{P}/[\cat{Q}])\). By abuse of notation we will let \(\rho\) be the induced functor \(\rho : \cat{D}^{[-d-1,0]} \rightarrow \hocat{[-d-1,0]}(\cat{P}/[\cat{Q}]) \).

	\begin{lemma}\label{lemma-rho-full-ess}
		Suppose \(\Extg{\cat{E}}^k(I,P) = 0\) for any \(1\leq k \leq d\), \(I \in \cat{I}\) and \(P \in \cat{P}\). Then the functor \(\rho\) is full and essentially surjective.
	\end{lemma}
	
	\begin{proof}
		We will say a morphism \(a\) in \(\cat{P}\) is a lift of a morphism \(\alpha\) in \(\cat{P}/[\cat{Q}]\) if \([a] = \alpha\). \par
		Let \(P^\bullet\) be a complex in \(\hocat{[-d-1,0]}(\cat{P}/[\cat{Q}])\) with differential \(\phi\). We say that \(P^\bullet\) satisfies condition \(\ast_k\) for some \(k \in \mathbb{Z}\) if there exist lifts \(f^n\) of \(\phi^n\) for all \(n\), satisfying \(f^{n+1} \circ f^{n} = 0\) for all \(n < k\), and these fit in a commutative diagram
		\begin{center}
		\begin{tikzcd}
			{P^{k-1}} && {P^{k}} && {P^{k+1}} & {P^{k+2}} \\
			& {C^{k-1}} && {C^{k}}
			\arrow["{f^{k-1}}", from=1-1, to=1-3]
			\arrow[from=1-1, to=2-2]
			\arrow["{f^{k}}", from=1-3, to=1-5]
			\arrow["x^k"', two heads, from=1-3, to=2-4]
			\arrow["{f^{k+1}}", from=1-5, to=1-6]
			\arrow[tail, from=2-2, to=1-3]
			\arrow["y^k"', from=2-4, to=1-5]
		\end{tikzcd}
		\end{center}
		in \(\cat{E}\), where \(C^{k-1}\inflation P^{k} \deflation C^{k}\) is a conflation, and \(\pd{C^{k}} \leq d + 1 + k\). \par
		Clearly, any complex in \(\hocat{[-d-1,0]}(\cat{P}/[\cat{Q}])\) satisfies \(\ast_{-d-1}\). Our goal is to show that a complex in \(\hocat{[-d-1,0]}(\cat{P}/[\cat{Q}])\) which satisfies \(\ast_k\) for some \(-d-1 \leq k < -1\) is isomorphic to a complex which satisfies \(\ast_{k+1}\). Since a complex that satisfies \(\ast_{-1}\) clearly lies in the image of \(\rho\), this will prove \(\rho\) is essentially surjective. \par
		So let \(P^\bullet\) be a complex in \(\hocat{[-d-1,0]}(\cat{P}/[\cat{Q}])\) which satisfies \(\ast_k\) for some fixed \(-d-1 \leq k < -1\). Since \(f^{k+1} \circ y^k \circ x^k = f^{k+1}\circ f^k \in [\cat{Q}]\), and as \(\pd{C^k}<d+1\), we can deduce from \Cref{lemma-Q-factor} that \(f^{k+1} \circ y^k \in [\cat{Q}]\). Let \(Q \in \cat{Q}\) and \(a: C^k \rightarrow Q\), \(b: Q \rightarrow P^{k+2}\) be such that \(b\circ a = f^{k+1}\circ y^k\). As \(\pd{C^k} < d+1\), we can use \Cref{corollary-pi-envelope} to assume \(a\) is an inflation. We define the complex
		\begin{align*}
			\cdots \xrightarrow{\phi^{k-1}} P^{k} \xrightarrow{\mat{\phi^k \\ [a \circ x^k]}} P^{k+1} \oplus Q \xrightarrow{\mat{\phi^{k+1} & [-b]}} P^{k+2} \xrightarrow{\phi^{k+2}} \cdots,
		\end{align*}
		which is isomorphic to \(P^\bullet\) in \(\hocat{[-d-1,0]}(\cat{P}/[\cat{Q}])\) and is easily verified to be satisfy \(\ast_{k+1}\). This proves \(\rho\) is essentially surjective. \par
		Now let \(P^\bullet\) and \(\tilde{P}^\bullet\) be two complexes with differentials \(f\) and \(\tilde{f}\) in \(\cat{D}^{[-d-1,0]}\), respectively. Let \(\psi^\bullet : \rho(P^\bullet) \rightarrow \rho(\tilde{P}^\bullet)\) be a morphism in \(\hocat{[-d-1,0]}(\cat{P}/[\cat{Q}])\). We say \(\psi^\bullet\) satisfies property \(\star_k\) for some \(k \in \mathbb{Z}\) if there exist lifts \(g^{n}\) of \(\psi^n\) for all \(n\) satisfying \(g^{n}\circ f^{n-1} = \tilde{f}^{n-1} \circ g^{n-1}\) for \(n \leq k\). \par
		Clearly any \(\psi^\bullet\) satisfies \(\star_{-d-1}\). If \(\psi^\bullet\) satisfies \(\star_0\), then it obviously lifts to an honest morphism of complexes in \(\hocat{[-d-1,0]}(\cat{P})\), and is hence in the image of \(\rho\). \par
		So let \(\psi^\bullet : \rho(P^\bullet) \rightarrow \rho(\tilde{P}^\bullet)\) be a morphism in \(\hocat{[-d-1,0]}(\cat{P}/[\cat{Q}])\) which satisfies \(\star_{k}\) for some \(-d-1 \leq k \leq -1\). According to \Cref{theorem-ess-image}, we can assume \(P^\bullet\) is exact up to degree \(k\) and \(\tilde{P}^\bullet\) is exact up to degree \(0\). We denote by \(C^{n-1}\xinflation{a^n} P^n \xdeflation{b^n} C^n\) and  \(\tilde{C}^{n-1}\xinflation{\tilde{a}^n} \tilde{P}^n \xdeflation{\tilde{b}^n} \tilde{C}^n\) the conflations that make up \(P^\bullet\) and \(\tilde{P}^\bullet\) for \(n \leq k\), respectively. \par
		Then locally around \(k\) the lifts \(g^n\) of \(\phi^n\) fit into the  following solid diagram
		\begin{center}
		\begin{tikzcd}
			{P^{k-1}} & {P^k} & {C^{k}} & {P^{k+1}} & {P^{k+2}} \\
			{\tilde{P}^{k-1}} & {\tilde{P}^k} & {\tilde{C}^k} & {\tilde{P}^{k+1}} & {\tilde{P}^{k+2}}
			\arrow["{f^{k-1}}", from=1-1, to=1-2]
			\arrow["{g^{k-1}}"', from=1-1, to=2-1]
			\arrow["{b^k}", two heads, from=1-2, to=1-3]
			\arrow["{g^{k}}"', from=1-2, to=2-2]
			\arrow["c", from=1-3, to=1-4]
			\arrow["h"', dashed, from=1-3, to=2-3]
			\arrow["{f^{k+1}}", from=1-4, to=1-5]
			\arrow["{g^{k+1}}"', from=1-4, to=2-4]
			\arrow["{g^{k+2}}"', from=1-5, to=2-5]
			\arrow["{\tilde{f}^{k-1}}"', from=2-1, to=2-2]
			\arrow["{\tilde{b}^{k}}"', two heads, from=2-2, to=2-3]
			\arrow["{\tilde{c}}"', from=2-3, to=2-4]
			\arrow["{\tilde{f}^{k+1}}"', from=2-4, to=2-5]
		\end{tikzcd},
		\end{center}
		where the left most square commutes, \(c\circ b^k=f^k\) and \(\tilde{c}\circ \tilde{b}^k=\tilde{f}^k\). Since \(\tilde{b}^k \circ g^{k} \circ f^{k-1} = 0\), there exists a morphism \(h : C^k \rightarrow \tilde{C}^{k}\) such that \(h \circ b^k = \tilde{b^k} \circ g^k\). We observe that 
		\begin{align*}
			(g^{k+1} \circ c-\tilde{c}\circ h)\circ b^k = g^{k+1} \circ f^k - \tilde{f}^{k} \circ g^k \in [\cat{Q}].
		\end{align*}
		Since \(\pd{C^k} < d+1\), we obtain from \Cref{lemma-Q-factor} that \(g^{k+1} \circ c-\tilde{c}\circ h \in [\cat{Q}]\). Let \(x : C^{k} \rightarrow Q\) and \(y : Q \rightarrow \tilde{P}^{k+1}\) be a  factorization of  \(g^{k+1} \circ a^{k+1}-\tilde{a}^{k+1}\circ h\) through some \(Q\in[\cat{Q}]\). By \Cref{corollary-pi-envelope}, since \(\pd{C^k}<d+1\) we can assume \(x\) is an inflation. We construct in \(\cat{D}^{[-d-1,0]}\) a complex \(\bar{P}^\bullet\) by
		\begin{align*}
			\cdots \xrightarrow{f^{k-1}} P^{k} \xrightarrow{\mat{f^k \\ b^k \circ x}} P^{k+1} \oplus Q \xrightarrow{\mat{f^{k+1} & 0}} P^{k+2} \xrightarrow{f^{k+2}} \cdots,
		\end{align*}
		which is exact up to \(k+1\). We define a morphism \(\bar{\psi}^\bullet : \rho(\bar{P}^\bullet) \rightarrow \rho(\tilde{P}^\bullet)\) by 
		\begin{align*}
			\bar{\psi}^n = \begin{cases}
				\mat{\psi^n & [y]}  & \text{if}\quad n = k+1, \\
				\psi^n & \text{if}\quad n \neq k.
			\end{cases}
		\end{align*}
		By construction \(\bar{\psi}^\bullet\) satisfies \(\star_{k+1}\). Finally, the morphism \(i^\bullet : \bar{P}^\bullet \rightarrow P^\bullet\) in \(\cat{D}^{[-d-1,0]}\) given by 
		\begin{align*}
			i^n = \begin{cases}
				\mat{\mathrm{id}_{P^n} & 0}  & \text{if}\quad n = k+1, \\
				\mathrm{id}_{P^n} & \text{if}\quad n \neq k,
			\end{cases}
		\end{align*}
		is an isomorphism. We observe that \(\bar{\psi}^\bullet = \psi^\bullet\circ \rho(i^\bullet)\). \par
		Repeating this argument recursively, we see that for any morphism \(\psi^\bullet : \rho(P^\bullet) \rightarrow \rho(\tilde{P}^\bullet)\) there exists a complex \(\bar{P}^\bullet\) in \(\cat{D}^{[-d-1,0]}\) together with an isomorphism \(i : \bar{P}^\bullet \rightarrow P^\bullet\), such that \(\psi^\bullet \circ \rho(i^\bullet) = \rho(j^\bullet)\) for some morphism \(j^\bullet\) in \(\cat{D}^{[-d-1,0]}\). In other words, the functor \(\rho\) is full. 
	\end{proof}
	
	\begin{proposition}\label{proposition-d-Aus}
		Suppose \(\Extg{\cat{E}}^k(I,P) = 0\) for any \(1\leq k \leq d\), \(I \in \cat{I}\) and \(P \in \cat{P}\). Then the functor  \(\rho\circ \Phi : \cat{E} \rightarrow \hocat{[-d-1,0]}(\cat{P}/[\cat{Q}]) \) is full and essentially surjective. Its kernel is the ideal \([\cat{I}\rightarrow \cat{P}]\) of morphisms factoring through a morphism with an injective domain and a projective codomain. 
	\end{proposition}
	
	\begin{proof}
		By \Cref{lemma-rho-full-ess}, the functor \(\rho\) is full and essentially surjective. The functor \(\Phi\) is an ideal quotient, so it is full and essentially surjective as well. The composition \(\rho\circ \Phi\) is thus full and essentially surjective. \par
		Let \(f : I \rightarrow P\) be a morphism in \(\cat{E}\) with \(I \in \cat{I}\) and \(P \in \cat{P}\). Then \(\rho\circ{\Phi(f)}\) is a morphism from \(\Sigma^{d+1} \tilde{P}\) to \(P\) in \(\hocat{[-d-1,0]}(\cat{P}/[\cat{Q}])\) for some \(\tilde{P} \in \cat{P}\). Then clearly \(\rho\circ{\Phi(f)} = 0\), and so \([\cat{I}\rightarrow \cat{P}] \subseteq \ker(\rho\circ\Phi)\). \par
		Conversely, let \(f : M \rightarrow N\) be a morphism in \(\cat{E}\) such that \(\rho\circ\Phi(f) = 0\). Let \(P^\bullet\) and \(\tilde{P}^\bullet\) be projective resolutions of \(M\) and \(N\) in \(\cat{E}\), respectively. Then \(f\) induces a morphism \(f^\bullet : P^\bullet \rightarrow \tilde{P}^\bullet\) in \(\hocat{[-d-1,0]}(\cat{P})\) such that \(f^k \in [\cat{Q}]\) for all \(k \in \mathbb{Z}\). According to \Cref{lemma-only-zero}, we can assume \(f^k = 0\) for all \(k \neq 0\). Let \(g : P^0 \inflation Q\) and \(h : Q \rightarrow \tilde{P}^0\) be a factorization of \(f^0\) through a projective-injective object \(Q\). Then the morphism \(f\) fits in the following solid commutative diagram whose rows are conflations
		\begin{center}
		\begin{tikzcd}
			K & {P^0} && M \\
			K & Q & E \\
			{\tilde{K}} & {\tilde{P}^0} && N
			\arrow[tail, from=1-1, to=1-2]
			\arrow[equals, from=1-1, to=2-1]
			\arrow[two heads, from=1-2, to=1-4]
			\arrow["g"', tail, from=1-2, to=2-2]
			\arrow[from=1-4, to=2-3]
			\arrow["f", from=1-4, to=3-4]
			\arrow[tail, from=2-1, to=2-2]
			\arrow["0"', from=2-1, to=3-1]
			\arrow[two heads, from=2-2, to=2-3]
			\arrow["h"', from=2-2, to=3-2]
			\arrow["s", dashed, from=2-3, to=3-2]
			\arrow[from=2-3, to=3-4]
			\arrow[tail, from=3-1, to=3-2]
			\arrow[two heads, from=3-2, to=3-4]
			\end{tikzcd}.
		\end{center}
		As \(\id{E} < d + 1\), \Cref{lemma-only-zero} implies the existence of the dashed morphism \(s \in [\cat{I}](E,\tilde{P}^0)\), which keeps the diagram commutative. Therefore \(f\) factors through a morphism with an injective domain and a projective codomain, as we wanted to show.
	\end{proof}
	
	\begin{theorem}\label{theorem-ideal-quotients}
		Let \(\cat{C}\) be an algebraic extriangulated category. Denote by \(\cat{P}\), \(\cat{I}\) and \(\cat{Q}\) its subcategories of projective, injective, and projective-injective objects, respectively. The following statements are equivalent:
		\begin{enumerate}[label = (\arabic*)]
			\item\label{d-Aus-1} \(\cat{C}\) admits an extriangulated ideal quotient which is equivalent as an extriangulated category to \(\hocat{[-d-1,0]}(\cat{A})\) for some additive category \(\cat{A}\).
			\item\label{d-Aus-2} \(\cat{C}\) is \(d\)-Auslander and satisfies \(\Ext{\cat{C}}^k(I,P) = 0\) for any \(1\leq k \leq d\), \(I \in \cat{I}\) and \(P \in \cat{P}\).
		\end{enumerate}
		The functors \(\Ext{\cat{C}}^k\) are as in \cite[Definition 3.6]{HLN-ii} (see also \cite[Section 3.1]{GNP-extensions}).
		When these equivalent conditions hold, we have \(\cat{A} \simeq \cat{P}/[\cat{Q}]\) and \(\cat{C}/[\cat{I}\rightarrow\cat{P}] \simeq \hocat{[-d-1,0]}(\cat{A})\).
	\end{theorem}
		
	\begin{proof}
		We let \(\cat{E}\) be an exact category with a subcategory \(\cat{Q}_0\) of projective-injective objects such that \(\cat{E}/[\cat{Q}_0] \simeq \cat{C}\). We let \(\tilde{\cat{P}}\), \(\tilde{\cat{I}}\) and \(\tilde{\cat{Q}}\) be its subcategories of projective, injective and projective-injective objects. It is clear that \(\cat{E}/[\tilde{\cat{I}}\rightarrow \tilde{\cat{P}}] \simeq \cat{C}/[\cat{I}\rightarrow\cat{P}]\) as extriangulated categories. The implication \((2) \implies (1)\) is then an immediate consequence of \Cref{proposition-d-Aus}. \par
		Suppose now that \(\cat{C}\) admits \(\hocat{[-d-1,0]}(\cat{A})\) as an extriangulated ideal quotient for some additive category \(\cat{A}\). Let \(\pi : \cat{C} \rightarrow \hocat{[-d-1,0]}(\cat{A})\) denote the quotient functor. The category \(\hocat{[-d-1,0]}(\cat{A})\) being \(d\)-Auslander forces \(\cat{C}\) to be \(d\)-Auslander. Moreover, we must have
		\begin{align*}
			\Ext{\cat{C}}^k(I,P) \simeq	\Extg{\hocat{[-d-1,0]}(\cat{A})}^k(\pi(I),\pi(P))
		\end{align*}
		for all \(1\leq k \leq d\), \(I \in \cat{I}\) and \(P \in \cat{P}\). But the latter vanishes trivially in \({\hocat{[-d-1,0]}(\cat{A})}\), so \((1) \implies (2)\). It is clear that \([\cat{I}\rightarrow\cat{P}] \subseteq \ker{\pi}\), which means that \(\pi\) induces the extriangulated ideal quotient
		\begin{align*}
			\tilde{\pi} : \cat{C}/[\cat{I}\rightarrow \cat{P}] \simeq \hocat{[-d-1,0]}(\cat{P}/[\cat{Q}]) \rightarrow {\hocat{[-d-1,0]}(\cat{A})}.
		\end{align*}
		To conclude \(\tilde{\pi}\) is an equivalence of extriangulated categories, we only need to show it is faithful. For ease of notation, we write \(\cat{K}_1 \coloneq \hocat{[-d-1,0]}(\cat{P}/[\cat{Q}]) \) and \(\cat{K}_2 \coloneq {\hocat{[-d-1,0]}(\cat{A})}\). \par
		Let \(X\) be an object in \(\cat{K}_1\), and consider a conflation \(X \inflation I \deflation Y\) with \(I\) an injective object. Let \(P\) be projective in \(\cat{K}_1\). Since \(\Hom{\cat{K}_1}(I,P) = 0\) and since \(\pi\) is an extriangulated ideal quotient, we have the following commutative diagram
		\begin{center}
		\begin{tikzcd}
			{\Hom{\cat{K}_1}(X,P)} & {\Ext{\cat{K}_1}(Y,P)} \\
			{\Hom{\cat{K}_2}(\pi(X),\pi(P))} & {\Ext{\cat{K}_2}(\pi(Y),\pi(P))}
			\arrow[hook, from=1-1, to=1-2]
			\arrow[two heads, from=1-1, to=2-1]
			\arrow[equals, from=1-2, to=2-2]
			\arrow[hook, from=2-1, to=2-2]
		\end{tikzcd},
		\end{center}
		In particular, we deduce that \({\Hom{\cat{K}_1}(X,P)} \simeq {\Hom{\cat{K}_2}(\pi(X),\pi(P))}\) for every \(X,P \in \cat{K}_1\) with \(P\) projective. \par
		Now for \(X \in \cat{K}_1\), consider a conflation \(Y' \inflation P \deflation X\) with \(P\) projective. By a similar reasoning, for an object \(Z \in \cat{K}_1\) we have the following commutative diagram
		\begin{center}
		\begin{tikzcd}[column sep=small]
			{\Hom{\cat{K}_1}(Z,Y')} & {\Hom{\cat{K}_1}(Z,P)} & {\Hom{\cat{K}_1}(Z,X)} & {\Ext{\cat{K}_1}(Z,Y')} \\
			{\Hom{\cat{K}_2}(\pi(Z),\pi(Y'))} & {\Hom{\cat{K}_2}(\pi(Z),\pi(P))} & {\Hom{\cat{K}_2}(\pi(Z),\pi(X))} & {\Ext{\cat{K}_2}(\pi(Z),\pi(Y'))}
			\arrow[from=1-1, to=1-2]
			\arrow[two heads, from=1-1, to=2-1]
			\arrow[from=1-2, to=1-3]
			\arrow[equals, from=1-2, to=2-2]
			\arrow[from=1-3, to=1-4]
			\arrow[two heads, from=1-3, to=2-3]
			\arrow[equals, from=1-4, to=2-4]
			\arrow[from=2-1, to=2-2]
			\arrow[from=2-2, to=2-3]
			\arrow[from=2-3, to=2-4]
		\end{tikzcd},
		\end{center}
		with exact rows. The four lemma then implies that \({\Hom{\cat{K}_1}(Z,X)} \simeq {\Hom{\cat{K}_2}(\pi(Z),\pi(X))}\) for any \(X,Z \in \cat{K}_1\). In particular \(\cat{P}/[\cat{Q}] \simeq \proj(\cat{K}_1) \simeq \proj(\cat{K}_2) \simeq \cat{A}\). This concludes our proof.
	\end{proof}
	
	\begin{example}
		Let \(\Lambda\) be the path algebra of the linearly oriented quiver of type \(A_3\) over a field \(k\). Let \(M\) be a basic additive generator of \(\mathrm{mod}(\Lambda)\), and set \(\Gamma = \mathrm{End}_{\Lambda}(M)\). From the classical Auslander correspondence (see \cite{Auslander, Iyama07}), it follows that \(\mathrm{mod}(\Gamma)\) is \(1\)-Auslander. It is easy to verify that \(\Extg{\Gamma}^1(D\Gamma,\Gamma) = 0\). We can thus apply \Cref{theorem-ideal-quotients} to obtain the extriangulated equivalence 
		\begin{align*}
			\mathrm{mod}(\Gamma)/[D\Gamma \rightarrow \Gamma] \simeq \hocat{[-2,0]}(\cat{A}),
		\end{align*}
		where \(\cat{A} = \mathrm{mod}(\Lambda)/[\mathrm{inj}(\mathrm{mod}(\Lambda))]\).
		
		\begin{figure}[ht]
			\begin{center}
				\def\svgwidth{0.96\textwidth}
				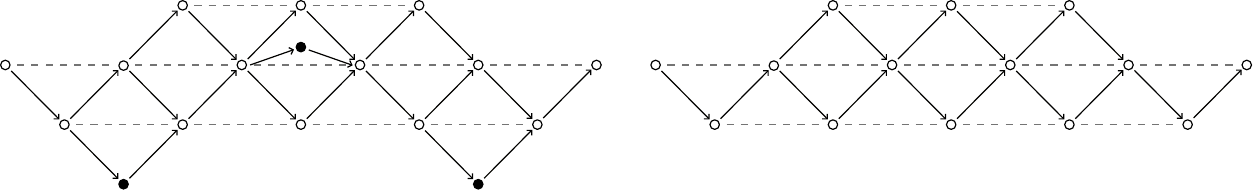
			\end{center}
			\caption{On the left, the Auslander--Reiten quiver of the module category \(\mathrm{mod}(\Gamma)\), the projective-injective objects represented by black nodes. On the right, the Auslander--Reiten quiver of the category of \(3\)-term complexes up to homotopy in \(\cat{A}\).}
		\end{figure}
	\end{example}

	\section{Cluster-tilting subcategories of triangulated categories}\label{triangulated}
	
	A class of \(d\)-Auslander extriangulated categories comes from triangulated categories \(\cat{T}\) with a \((d+2)\)-cluster-tilting subcategory \(\cat{M}\). Concretely, we consider the category \(\cat{T}_\cat{M}\), whose underlying additive category is \(\cat{T}\), and its extriangulated structure is the relative structure making \(\cat{M}\) the subcategory of projective objects. \par
	The following result is a higher-dimensional generalization of \cite[Section 3.3.6]{GNP23} (which corresponds to the case \(d=0\)).
		
	\begin{proposition}\label{proposition-CT-d-Auslander}
		The category \(\cat{T}_\cat{M}\) is \(d\)-Auslander.
	\end{proposition}
	
	\begin{proof}
		We begin by emphasizing that a triangle \(X \rightarrow Y \rightarrow Z \xrightarrow{\delta} \Sigma X\) in \(\cat{T}\) is a conflation \(X \inflation Y \deflation Z\) in \(\cat{T}_\cat{M}\) if and only if the natural transformation  \((M,Z) \xrightarrow{\delta\circ-} (M,X)\) is zero for all \(M \in \cat{M}\). This is equivalent to \(\delta\) factoring through an object in \(\Sigma\cat{M} \ast \cdots \ast \Sigma^{d+1}\cat{M}\). \par
		We remark that any triangle in \(\cat{T}\) of the form 
		\begin{align*}
			X \rightarrow M \xrightarrow{f} Z \rightarrow \Sigma X
		\end{align*}
		with \(f\) a right \(\cat{M}\)-approximation satisfies this property. \par
		Let \(Z\) be any object in \(\cat{T}\). As \(\cat{M}\) is \((d+2)\)-cluster-tilting, for any \(0 \leq k \leq d\) there is a triangle  
		\begin{align}\label{triangles-resolution}
			Z_{k+1} \rightarrow M_k \xrightarrow{f_k} Z_k \rightarrow \Sigma Z_{k+1}
		\end{align}
		in \(\cat{T}\) such that \(Z_0 = Z\), the morphism \(f_k\) is a right \(\cat{M}\)-approximation for each \(k\), and \(Z_{d+1}\) belongs to \(\cat{M}\). In particular, in \(\cat{T}_\cat{M}\) we obtain the sequence of conflations
		\begin{align*}
			Z_{k+1} \inflation M_k \xdeflation{f_k} Z_k
		\end{align*}
		which amounts to a projective resolution of \(Z\). In particular \(\cat{T}_\cat{M}\) has enough projectives and \(\cat{T}_\cat{M}\) has global dimension at most \(d+1\). By construction a projective resolution of \(\Sigma^{d+1} M\), we see that \(\pd{\Sigma^{d+1} M} = d+1\) for any nonzero \(M \in \cat{M}\), so \(\cat{T}_\cat{M}\) has global dimension \(d+1\). \par
		We now want to show that \(\inj(\cat{T}_\cat{M}) = \Sigma^{d+1} \cat{M}\). Let us consider a conflation \(X \inflation Y \deflation Z\) in \(\cat{T}_\cat{M}\), which by definition comes from a triangle
		\begin{align*}
			\Sigma^{-1} Z \xrightarrow{\theta} X \rightarrow Y \rightarrow Z
		\end{align*}
		in \(\cat{T}\) such that \(\theta\) factors through \(\cat{M} \ast \cdots \ast \Sigma^{d} \cat{M}\). This means that any morphism from \(X\) to an object in \(\Sigma^{d+1} \cat{M}\) must factor through \(Y\), so \(\Sigma^{d+1} \cat{M} \subseteq \inj(\cat{T}_\cat{M})\). We notice that \(X \inflation 0 \deflation \Sigma X\) is a conflation in \(\cat{T}_\cat{M}\) for any \(X \in \cat{M} \ast \cdots \ast \Sigma^{d} \cat{M}\). In particular, any injective object in \(\cat{T}_\cat{M}\) must belong to \(\Sigma^{d+1} \cat{M}\), which establishes  \(\inj(\cat{T}_\cat{M}) = \Sigma^{d+1} \cat{M}\). \par
		Finally, constructing the triangles \(\eqref{triangles-resolution}\) for \(Z\) in \(\Sigma^{d+1} \cat{M}\), we see that \(M_k = 0\) for all \(0 \leq k \leq d\) and so the dominant dimension in \(\cat{T}_\cat{M}\) is \(d+1\). 
		Therefore \(\cat{T}_\cat{M}\) is \(d\)-Auslander.
	\end{proof}
	
	\begin{corollary}\label{corollary-CT-vosnex}
		Let \(\cat{T}\) be an algebraic triangulated category with a \((d+2)\)-cluster-tilting subcategory \(\cat{M}\) which satisfies the vosnex property. Then
		\begin{align*}
			\cat{T}_\cat{M}/[\Sigma^{d+1}\cat{M}\rightarrow \cat{M}] \simeq \cat{K}^{[-d-1,0]}(\cat{M})
		\end{align*}
		as extriangulated categories.
	\end{corollary}
	
	\begin{proof}
		The extriangulated category \(\cat{T}_\cat{M}\) is algebraic by \cite[Corollary 4.37]{ChenDGi}. By \Cref{proposition-CT-d-Auslander}, it is \(d\)-Auslander. \(\cat{M}\) is its subcategory of projective objects, and \(\Sigma^{d+1}\cat{M}\) is its subcategory of injective objects. Since \(\cat{M}\) has the vosnex property, condition \((2)\) of \Cref{theorem-ideal-quotients} holds, so 
		\begin{align*}
			\cat{T}_\cat{M}/[\Sigma^{d+1}\cat{M}\rightarrow \cat{M}] \simeq \cat{K}^{[-d-1,0]}(\cat{M})
		\end{align*}
		as extriangulated categories. 
	\end{proof}
	
	\begin{corollary}\label{corollary-dZCT}
		Let \(\cat{T}\) be an algebraic triangulated category with a \((d+2)\mathbb{Z}\)-cluster-tilting subcategory \(\cat{M}\). Then \(	\cat{T}_\cat{M} \simeq \cat{K}^{[-d-1,0]}(\cat{M})\) as extriangulated categories.
	\end{corollary}
	
	\begin{proof}
		Obviously \(\cat{M}\) satisfies the vosnex property, so \Cref{corollary-CT-vosnex} implies 
		\begin{align*}
			\cat{T}_\cat{M}/[\Sigma^{d+1}\cat{M}\rightarrow \cat{M}] \simeq \cat{K}^{[-d-1,0]}(\cat{M}).
		\end{align*}
		But \(\cat{M}\) being \((d+2)\mathbb{Z}\)-cluster-tilting implies \([\Sigma^{d+1}\cat{M}\rightarrow \cat{M}] = 0\), and the claim follows.
	\end{proof}
	
	\begin{example}
		Let \(\Lambda\) be the path algebra of the linearly oriented quiver of type \(A_{d+3}\) over a field \(k\). Combining \cite[Proposition 2.1]{HappelSeidel} and \cite[Theorem 1.22]{Iyama11}, we see that \(\cat{T}\coloneq \cat{D}^b(\Lambda)\) contains a \((d+2)\mathbb{Z}\)-cluster-tilting subcategory \(\cat{U}\), which can be described as follows. Denote by \(P\) the simple projective object in \(\mathrm{mod}(\Lambda)\), then \(\cat{U} = \{(\tau\Sigma)^i P \mid i \in \mathbb{Z}\}\), where \(\tau\) is the Auslander-Reiten translation. Thus \eqref{corollary-dZCT} implies that there is an extriangulated equivalence
		\begin{align*}
			\cat{T}_\cat{U} \simeq \cat{K}^{[-d-1,0]}(\cat{U}). 
		\end{align*}
		\begin{figure}[ht]
			\begin{center}
				\def\svgwidth{0.7\textwidth}
				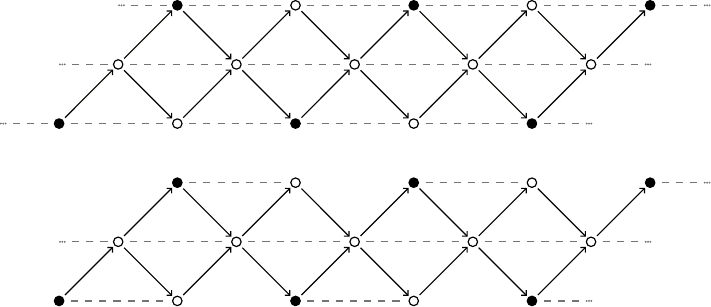
			\end{center}
			\caption{The case \(d=0\). On the top, the Auslander--Reiten quiver of the derived category \(\cat{D}^b(\Lambda)\), the subcategory \(\cat{U}\) represented by black nodes. On the bottom, the Auslander--Reiten quiver of the category of \(2\)-term complexes up to homotopy in \(\cat{U}\).}
		\end{figure}
	\end{example}
	
	We can now give a solution to a question raised by Iyama in \cite{Yang}.
	
	\begin{corollary}\label{cor-angulated}
		Let \(\cat{N}\) be a weakly idempotent complete algebraic \((d+4)\)-angulated category. Then \(\hocat{[-d-1,0]}(\cat{N})\) has a structure of a triangulated category. 
	\end{corollary}
	
	\begin{proof}
		\cite[Theorem 7.5]{Kvamme} implies that \(\cat{N}\) is additively equivalent to a \((d+2)\mathbb{Z}\)-cluster-tilting subcategory \(\cat{M}\) of an algebraic triangulated category \(\cat{T}\). Then \Cref{corollary-dZCT} implies
		\begin{align*}
			\hocat{[-d-1,0]}(\cat{N}) \simeq \hocat{[-d-1,0]}(\cat{M}) \simeq \cat{T}_\cat{M},
		\end{align*}
		so \(\hocat{[-d-1,0]}(\cat{N})\) has a structure of a triangulated category.
	\end{proof}
	
	Using the above result, we can give a generalization to Theorem A.1 of \cite[Appendix A]{Yang}.
	
	\begin{theorem}
		Let \(k\) be a commutative ring. Let \(\cat{T}\) be an algebraic, idempotent complete, \(k\)-linear, Hom-finite, locally finite triangulated category. Let \(\cat{M}\) be a \((d+2)\)-cluster-tilting subcategory of \(\cat{T}\), which satisfies the vosnex property. Then the following conditions are equivalent.
		\begin{enumerate}
			\item \(\hocat{[-d-1,0]}(\cat{M})\) has a structure of a triangulated category.
			\item The category \(\mo \cat{M}\) is Frobenius. 
			\item \(\cat{M}\) is \((d+2)\mathbb{Z}\)-cluster-tilting in \(\cat{T}\).
			\item \(\cat{M}\) has a structure of an algebraic \((d+4)\)-angulated category. 
		\end{enumerate}
	\end{theorem}
	
	We note that the assumptions on \(\cat{T}\) imply that the categories \(\cat{M}\), \(\mo \cat{M}\) and \(\hocat{[-d-1,0]}(\cat{M})\) are Krull-Schmidt. Moreover, \(\mo \cat{M}\) is an abelian length category. 
	
	\begin{proof}
		We prove \((1) \implies (2)\). This is an adaptation of the proof of \cite[Theorem A.2]{Yang} . We remark that \(\hocat{[-d-1,0]}(\cat{M}^\mathrm{op})\) admits a triangulated structure, since
		\begin{align*}
			\hocat{[-d-1,0]}(\cat{M}^\mathrm{op}) \simeq \hocat{[-d-1,0]}(\cat{M})^\mathrm{op},
		\end{align*}
		and the opposite category of a triangulated category is triangulated. Denote the shift functor for these triangulated structures by \(\tilde{\Sigma}\). \par
		Let \(S\) be a non-projective simple object in \(\mo \cat{M}\) and let 
		\begin{align*}
			(-,P^{-d-3}) \xrightarrow{f^{-d-3}\circ-} (-,P^{-d-2}) \xrightarrow{f^{-d-2}\circ-} \cdots \xrightarrow{f^{-1}\circ-} (-,P^0) \deflation S
		\end{align*}
		be a minimal projective resolution. By \cite[Lemmas A.4-A.7]{Yang} (which generalize immediately to our setting), we have the following almost split triangle in \(\hocat{[-d-1,0]}(\cat{M})\)
		\begin{align*}
			\Sigma^{d+3} P^{-d-3} \xrightarrow{\tilde{f}^{-d-3}} C \xrightarrow{\tilde{f}^{-1}} {P^{0}} \rightarrow \tilde{\Sigma}\Sigma^{d+3} P^{-d-3},
		\end{align*}
		where \(C\) is the complex \(P^{-d-2} \xrightarrow{f^{-d-2}} \cdots \xrightarrow{f^{-2}} P^{-2}\), and the morphisms \(\tilde{f}^{-d-3}\) and \(\tilde{f}^{-1}\) are given by
		\begin{center}
		\begin{tikzcd}[row sep = scriptsize]
			{P^{-d-3}} & {P^{-d-2}} & 0 \\
			\vdots & \vdots & \vdots \\
			0 & {P^{-1}} & {P^0}
			\arrow["{f^{-d-3}}", from=1-1, to=1-2]
			\arrow[from=1-1, to=2-1]
			\arrow[from=1-2, to=1-3]
			\arrow["{f^{-d-2}}", from=1-2, to=2-2]
			\arrow[from=1-3, to=2-3]
			\arrow[from=2-1, to=3-1]
			\arrow["{f^{-2}}", from=2-2, to=3-2]
			\arrow[from=2-3, to=3-3]
			\arrow[from=3-1, to=3-2]
			\arrow["{f^{-1}}", from=3-2, to=3-3]
		\end{tikzcd}.
		\end{center}
		By duality, we obtain an almost split sequence in \(\hocat{[-d-1,0]}(\cat{M}^\mathrm{op})\), which by \cite[Lemma A.5]{Yang} gives the exact sequence
		\begin{align*}
			(P^{0},-) \xrightarrow{-\circ f^{-1}} \cdots \xrightarrow{-\circ f^{-d-2}} (P^{-d-2},-) \xrightarrow{-\circ f^{-d-3}}(P^{-d-3},-)
		\end{align*}
		in \(\mo \cat{M}^\mathrm{op}\). In particular, we see that \(\Extg{\cat{M}}^1(S,(-,P)) = 0\) for any \(P\) in \(\cat{M}\) and simple \(S\) in \(\mo \cat{M}\). Since \(\mo \cat{M}\) is a length category, this implies that \((-,P)\) is an injective object, hence \(\mo \cat{M}\) is Frobenius. \par
		The implication \((2) \implies (3)\) is \cite[Corollary 5.8 (i)]{BELIGIANNIS}. See also \cite[Proposition 3.6]{IO13}. \par
		The implication \((3) \implies (4)\) follows from \cite[Theorem 1]{GKO} and \cite[Section 6.5]{Jas} (see also in \cite[Theorems 5.14 and 5.16]{Jas}). \par
		Finally, \((4) \implies (1)\) follows from \Cref{cor-angulated}.
	\end{proof}
	
	\section{Connection to higher Auslander algebras}\label{higher}
	
	Let \(k\) be a field and \(\Gamma\) be a finite dimensional \(k\)-algebra. Assume \(\Gamma\) is \(d\)-Auslander and satisfies the equivalent conditions of \Cref{theorem-ideal-quotients}, that is, \(\Extg{\Gamma}^k(D\Gamma,\Gamma) = 0\) for \(1 \leq k \leq d\). Then according to \cite[Theorem 1.20]{Iyama11}, this is equivalent to the existence of a finite dimensional \(k\)-algebra \(\Lambda\) with \(\mathrm{gl.dim}\ \Lambda \leq d\) and a \(d\)-cluster-tilting object \(M\) in \(\mo \Lambda\) such that \(\Gamma \simeq \mathrm{End}_\Lambda(M)\). In this setting, \Cref{theorem-ideal-quotients} has the following immediate consequence.
	
	\begin{corollary}
		We have an equivalence of extriangulated categories
		\begin{align*}
			\mathrm{mod}(\Gamma) / [D\Gamma \rightarrow \Gamma] \simeq \hocat{[-d-1,0]}(\add(M)/[\mathrm{inj}(M)]).
		\end{align*}
	\end{corollary}
	
	In particular, there is a bijection between \((d+1)\)-cluster-tilting subcategories of \(\mathrm{mod}(\Gamma)\) and \(\hocat{[-d-1,0]}(\add(M)/[\mathrm{inj}(M)])\). But according to \cite{Iyama11}, there is a \((d+1)\)-cluster-tilting subcategory \(\cat{M}\) in \(\mathrm{mod}(\Gamma)\) if and only if \(\Gamma\) is a higher Auslander algebra of type \(A\), which in this case is unique. In particular, we obtain an equivalence of extriangulated categories
	\begin{align*}
		\mathrm{mod}(A_{n+1}^{d+1}) / [DA_{n+1}^{d+1} \rightarrow A_{n+1}^{d+1}] \simeq \hocat{[-d-1,0]}(\mathrm{add}(M^{d}_{n+1})/[\mathrm{inj}(M^{d}_{n+1})]).
	\end{align*}
	But \(\mathrm{add}(M^{d}_{n+1})/[\mathrm{inj}(M^{d}_{n+1})] \simeq \mathrm{add}(M^{d}_{n}) \simeq \mathrm{proj}(A^{d+1}_{n})\), so we get
	\begin{align*}
		\mathrm{mod}(A_{n+1}^{d+1}) / [DA_{n+1}^{d+1} \rightarrow A_{n+1}^{d+1}] \simeq \hocat{[-d-1,0]}(\mathrm{proj}(A^{d+1}_{n})).
	\end{align*}
	Restricting this equivalence to the unique \((d+1)\)-cluster-tilting subcategories on each side, we obtain
	\begin{align}\label{equivalence-GW}
		\mathrm{add}(M^{d+1}_{n+1})/[DA_{n+1}^{d+1} \rightarrow A_{n+1}^{d+1}] \simeq \tilde{\cat{P}}_{d+1}(\add({M}_n^{d+1})),
	\end{align}
	where \(\tilde{\cat{P}}_{d+1}(\add({M}_n^{d+1}))\) is the category in \cite[Theorem 7.3]{AHJKPT}. We notice that \eqref{equivalence-GW} is exactly the equivalence of  of \cite[Theorem 3.1]{GW}. The discussion above suggests that this equivalence is a special phenomenon that is unique to the higher Auslander algebras of type \(A\). 
	
	\section*{References}
	\addcontentsline{toc}{section}{References}
	\printbibliography[heading=none]

	\bigskip
	\bigskip
	
	Lior Silberberg, \textsc{Université Paris Cité, Sorbonne Université, CNRS, IMJ-PRG, F-75013 Paris, France} \par
	\textit{E-mail address}: \texttt{lior.silberberg@imj-prg.fr}

\end{document}